\documentclass[12pt]{article}
\usepackage{amsmath,amsthm,amsfonts,amssymb}
\usepackage{fullpage}
\usepackage{multirow}
\usepackage{array}
\usepackage{bbm}
\usepackage[noblocks]{authblk}
\usepackage{verbatim}
\usepackage{tikz}
\usepackage{float}
\usepackage{enumerate}
\usepackage{diagbox}
\usepackage{soul}
\usepackage{url}
\usepackage{mathtools}
\usepackage{hyperref}
\usepackage{listings}

\newtheorem{theorem}{Theorem}[section]

\newtheorem{lemma}[theorem]{Lemma}

\newtheorem{corollary}[theorem]{Corollary}

\theoremstyle{definition}

\newlength{\Oldarrayrulewidth}

\newcommand{\Z}{\mathbb{Z}}

\newcommand{\p}{\textup{\textbf{p}}}

\begin{document}

\author[1]{Chris~Bispels\thanks{cbispel1@umbc.edu}}
\author[2]{Matthew~Cohen\thanks{matthewcohen@cmu.edu}}
\author[3]{Joshua~Harrington\thanks{joshua.harrington@cedarcrest.edu}}
\author[4]{Joshua~Lowrance\thanks{joshua.lowrance@biola.edu}}
\author[5]{Kaelyn~Pontes\thanks{kaelyn.pontes@hastings.edu}}
\author[6]{Leif~Schaumann\thanks{schaumann1@kenyon.edu}}
\author[8]{Tony~W.~H.~Wong\thanks{wong@kutztown.edu}}
\affil[1]{Department of Mathematics, University of Maryland, Baltimore County}
\affil[2]{Department of Mathematical Sciences, Carnegie Mellon University}
\affil[3]{Department of Mathematics, Cedar Crest College}
\affil[4]{Department of Mathematics and Computer Science, Biola University}
\affil[5]{Department of Mathematics, Hastings College}
\affil[6]{Department of Mathematics, Kenyon College}
\affil[7]{Department of Mathematics, Kutztown University of Pennsylvania}
\date{\today}

\title{A further investigation on covering systems with odd moduli}
\maketitle

\begin{abstract}
Erdős first introduced the idea of covering systems in 1950. Since then, much of the work in this area has concentrated on identifying covering systems that meet specific conditions on their moduli. Among the central open problems in this field is the well-known odd covering problem. In this paper, we investigate a variant of that problem, where one odd integer is permitted to appear multiple times as a modulus in the covering system, while all remaining moduli are distinct odd integers greater than 1.

\end{abstract}

\section{Introduction}\label{sec:introduction}
The concept of covering systems of the integers, first introduced by Paul Erdős in 1950 \cite{erdos}, has sparked extensive research over the last 75 years. A \emph{covering system of the integers}, or a \emph{covering system} for short, is a finite collection of congruences where every integer satisfies at least one of the congruences in the set. One of the most notable open questions in this area, posed by Erdős, is whether there exists an odd covering system. An \textit{odd covering system} is a covering system where all moduli are distinct odd integers greater than 1. The question of the existence of such a covering system has become known as \textit{the odd covering problem}. Recent developments on the odd covering problem have shown that if an odd covering system exists, then the least common multiple of its moduli must be divisible by $9$ or $15$ \cite{bbmst}. 

In recent years, variations of the odd covering problem have been explored. We say that a subset of the integers is \emph{covered} by a set of congruences if every element of the set satisfies at least one of the congruences. We say that a subset of the integers has an \emph{odd covering} if it can be covered by a finite set of congruences whose moduli are distinct odd integers greater than 1. An investigation by Filaseta and Harvey in 2018 \cite{fh} showed that each of the following sets has an odd covering: prime numbers, numbers that can be written as the sum of two squares, and Fibonacci numbers.

Another variation on the odd covering problem asks, for a given odd prime $p$, what the smallest nonnegative integer $t_p$ is such that there exists a covering system of the integers where $p$ is a modulus in $t_p$ congruences while all other moduli are distinct odd integers greater than $1$. If an odd covering system exists, then $t_p\leq1$ for all odd primes $p$. Table~\ref{table:original} shows the known bounds for $t_p$ to date.
\begin{table}[H]
\centering
\begin{tabular}{|c|c|c|}
\hline
$p$& Upper bound on $t_p$& Source\\
\hline
$3$& $t_p\leq2=p-1$& \cite{harrington}\\
$5$& $t_p\leq3=p-2$& \cite{hhm}\\
$7$& $t_p\leq4=p-3$& \cite{hsw}\\
$11\leq p\leq19$& $t_p\leq p-4$& \cite{hsw}\\
$p\geq23$& $t_p\leq p-5$& \cite{hsw}\\
\hline
\end{tabular}
\caption{Bounds on $t_p$}
\label{table:original}
\end{table}

In this article, we directly improve some of the bounds provided in Table~\ref{table:original}. Moreover, we loosen the restriction above to consider $t_k$ when $k$ is an odd integer, not necessarily a prime. That is, for an odd integer $k$, we let $t_k$ be the smallest nonnegative integer such that there exists a covering system of the integers where $k$ is a modulus in $t_k$ congruences while all other moduli are distinct odd integers greater than 1. Our main results are summarized by Table~\ref{table:improvedt_kbounds}.
\begin{table}[H]
\centering
\begin{tabular}{|c|c|c|}
\hline
$k$ & Upper bound on $t_{k}$&Source \\
\hline
$9$ & $t_{9} \leq 3$&Theorem~\ref{thm:three9s} \\
$15$ & $t_{15}\leq4$&Theorem~\ref{thm:four15s}\\
$21$ & $t_{21}\leq5$&Theorem~\ref{thm:five21s}\\
$25$ & $t_{25}\leq 8$&Theorem~\ref{thm:eight25s} \\
$49$ & $t_{49}\leq 22$&Theorem~\ref{thm:splittingn}\\
$k=p$ where $p\geq 17$ is prime&$t_k\leq p-5$&Theorem~\ref{thm:p-5}\\
$k=p^2$ where $11\leq p\leq 13$ is prime&$t_k\leq p(p-5)$&Corollary~\ref{cor:improving11and13squared}\\
$k=p^2$ where $p\geq 17$ is prime& $t_{k} \leq p(p - 6)$&Theorems~\ref{thm:splittingn} and \ref{thm:p-5} \\
\hline
\end{tabular}
\caption{Bounds on $t_k$ achieved in this article}
\label{table:improvedt_kbounds}
\end{table}

An immediate application of our results on odd covering systems with repeated moduli is to further Filaseta and Harvey's investigation by demonstrating the existence of an odd covering for various important subsets of the integers. Indeed, if $\mathcal{C}_{k,t}=\{r_i\pmod{k}:1\leq i\leq t\}\cup \mathfrak{C}$ is a covering system, where $\mathfrak{C}$ is a finite collection of congruences with moduli that are distinct odd integers greater than 1 and not equal to $k$, then the set of congruences $\{r_t\pmod{k}\}\cup\mathfrak{C}$ is an odd covering of the set $\{a\in\mathbb{Z}:a\not\equiv b\pmod{k}\text{ for }b\in\{r_1,\ldots,r_{t-1}\}\}$. By extending this idea in Section~\ref{sec:coveringsubsets}, we establish the following theorem.

%We note that the improvements in the cases when $k$ is prime contribute to Filaseta and Harvey's investigation on odd coverings of subsets of the integers. Indeed, Hammer, Harrington, and Marrotta showed that for an odd prime $p$, if all terms of a given integer sequence satisfy at most $p-t_p+1$ congruence classes modulo $p$, then it has an odd covering \cite{hhm}. Further, the covering system that we present in the proof of Theorem~\ref{thm:three9s} to establish $t_9\leq 3$ will also establish the following theorem.

\begin{theorem}\label{thm:integersequences}
The union of the following subsets of the integers has an odd covering:
\begin{itemize}
\item Sums of two squares (OEIS: A001481),
\item Sums of two cubes (OEIS: A045980),
\item Powerful numbers (OEIS: A001694),
\item Primes and powers of primes (OEIS: A000961),
%\item Factorials(OEIS: A000142)
\item Numbers of derangements (OEIS: A000166),
\item Fermat numbers (OEIS: A000215)
%\item Mersenne numbers (OEIS: A000225)
\item Perfect numbers (OEIS: A000396).
\end{itemize}
\end{theorem}
%%%%%%%%%%%%%%% I took factorials and Fermat numbers off of the list because those sequences are covered by the covering given in \cite{harrington}. Specifically, we should focus on sequences that (may) contain infinitely many terms in each of the congruence classes modulo 3 but avoids 3\pmod{9} and 6\pmod{9}.

%%%%%%%%%%%%%%%%% I don't know if we should include a statement on why some of these sets works. 

%%%%%%%%% It is know that all perfect numbers are either of the form $2^{p-1}(2^p-1)$ where $p$ is prime or it must satisfy one of $1\pmod{12}$, $117\pmod{468}$, or $81\pmod{324}$. Since all but finitely many primes satisfy either $p\equiv 1\pmod{6}$ or $p\equiv 5\pmod{6}$, it follows that with the exception of 6, any perfect number that is divisible by 3 must also be divisible by 9. Notice here that if there are no odd perfect numberss, then the perfect numbers are covered by the covering system in \cite{harrington}.

%%%%%%%%% One can show that the sequence of derangements numbers is periodic modulo 3 with a period of 6. From there, it is easy to show that if a derangement number is divisible by 3, then it must be divisible by 9. 

\section{Main Results}
We begin this section with a lemma derived from a result of Filaseta and Harvey \cite{fh}, which is helpful in simplifying the proofs of various theorems in this article. 

\begin{lemma}\label{lem:coveringshift}
Let $\mathcal{C}=\{r_i\pmod{m_i}:1\leq i\leq\upsilon\}$ be a covering system and let $j$ be an integer. Then $\mathcal{C}'=\{r_i+j\pmod{m_i}:1\leq i\leq\upsilon\}$ is also a covering system.
\end{lemma}

The next theorem establishes odd covering systems with repeated moduli based on the existence of others.

\begin{theorem}\label{thm:splittingn}
For any positive integer $t$ and odd integer $k\geq 3$, if there exists a covering system $\mathcal{C}_{k,t}$ with moduli that are odd, greater than 1, and distinct except that the modulus $k$ is used $t$ times, then for any integer $m\geq 2$, there exists a covering system with moduli that are odd, greater than 1, and distinct except that the modulus $km$ is used at most $m(t-1)+1$ times. Further, if $\gcd(k,m)=1$, then there exists a covering system with moduli that are odd, greater than 1, and distinct except that the modulus $km$ is used at most $(m-1)(t-1)+1$ times. In both cases, if $\mathcal{C}_{k,t}$ does not have $km$ as a modulus, then each of these bounds can be reduced by $1$.
\end{theorem}
\begin{proof}
Let $\mathcal{C}_{k,t}=\{r_i\pmod{k}:1\leq i\leq t\}\cup \mathfrak{C}$, where $\mathfrak{C}$ is a collection of congruences with moduli not equal to $k$. Notice that for each $1\leq i\leq t-1$, any integer satisfying $r_i\pmod{k}$ also satisfies $kj+r_i\pmod{km}$ for some $0\leq j\leq m-1$. Therefore, $\mathcal{C}=\{kj+r_i\pmod{km}:1\leq i\leq t-1, 0\leq j\leq m-1\}\cup\{r_t\pmod{k}\}\cup \mathfrak{C}$ is a covering system.

Next, suppose the additional condition $\gcd(k,m)=1$, so that $k$ has a multiplicative inverse modulo $m$. Using Lemma~\ref{lem:coveringshift}, we assume without loss of generality that if $m$ is the modulus of a congruence in $\mathfrak{C}$, then $0\pmod{m}$ is a congruence in $\mathfrak{C}$. Now, $\mathcal{C}=\{0\pmod{m}\}\cup\{kj+r_i\pmod{km}:1\leq i\leq t-1, 0\leq j\leq m-1, j\not\equiv-k^{-1}r_i\pmod{m}\}\cup\{r_t\pmod{k}\}\cup \mathfrak{C}$ is a covering system. In either case, since $\mathfrak{C}$ may have a congruence with modulus $km$, $\mathcal{C}$ is a covering system with moduli that are odd, greater than 1, and distinct except that the modulus $km$ is used at most the desired number of times.
\end{proof}

We note here that the covering systems constructed by Harrington, Sun, and Wong \cite{hsw} to show that $t_p\leq p-4$ for $p\in\{11,13\}$ do not have $p^2$ as a modulus. Consequently, our next corollary follows from Theorem~\ref{thm:splittingn}.

\begin{corollary}\label{cor:improving11and13squared}
For $p\in\{11,13\}$, there exists a covering system of the integers such that all moduli are odd, greater than 1, and distinct except that the modulus $p^2$ is used at most $p(p-5)$ times. 
\end{corollary}

Theorem~\ref{thm:splittingn} provides bounds on $t_k$ for all composite $k$ using the results in Table~\ref{table:original}. We will see that, in certain cases, these bounds can be improved by direct construction of such covering systems.

%For the remainder of this paper, we use the tree diagram notation established by Harrington, Sun, and Wong \cite{hsw} to present our covering systems.

To present these covering systems, we adopt the tree diagram notation established by Harrington, Sun, and Wong \cite{hsw} with minor adjustments and additions. Recall that when we write $\{m_1,m_2,\dotsc,m_\ell\}\times m_0$, it indicates a list of $2^\ell$ moduli, given by the product of any subset of $\{m_1,m_2,\dotsc,m_\ell\}$ together with $m_0$. In this paper, these moduli are sorted from left to right as follows.
$$\footnotesize\begin{tabular}{rrrr}
$m_0$,& $m_1\times m_0$,& $m_2\times m_0$,& $m_1\times m_2\times m_0$,\\
$m_3\times m_0$,& $m_1\times m_3\times m_0$,& $m_2\times m_3\times m_0$,& $m_1\times m_2\times m_3\times m_0$,\\
\vdots& \vdots& \vdots& \vdots\\
$m_3\times\dotsb m_\ell\times m_0$,&  $m_1\times m_3\times\dotsb m_\ell\times m_0$,& $m_2\times m_3\times\dotsb m_\ell\times m_0$,& $m_1\times m_2\times m_3\times\dotsb m_\ell\times m_0$.
\end{tabular}$$
Under some wedges, we have new notation $[m_1^\alpha]\times\{m_2,m_3,\dotsc,m_\ell\}\times m_0$. This indicates a list of $(\alpha+1)2^{\ell-1}$ moduli, given by the product of the following: an element in $\{1,m_1,m_1^2,\dotsc,m_1^\alpha\}$, a subset of $\{m_2,m_3,\dotsc,m_\ell\}$, and $m_0$. These moduli are sorted from left to right as follows.
$$\footnotesize\begin{tabular}{rrrr}
$m_0$,& $m_1\times m_0$,& $\dotsc$,& $m_1^\alpha\times m_0$,\\
$m_2\times m_0$,& $m_1\times m_2\times m_0$,& $\dotsc$,& $m_1^\alpha\times m_2\times m_0$,\\
$m_3\times m_0$,& $m_1\times m_3\times m_0$,& $\dotsc$,& $m_1^\alpha\times m_3\times m_0$,\\
$m_2\times m_3\times m_0$,& $m_1\times m_2\times m_3\times m_0$,& $\dotsc$,& $m_1^\alpha\times m_2\times m_3\times m_0$,\\
\vdots& \vdots& & \vdots\\
$m_2\times m_3\dotsb\times m_\ell\times m_0$,& $m_1\times m_2\times m_3\dotsb\times m_\ell\times m_0$,& $\dotsc$,& $m_1^\alpha\times m_2\times m_3\dotsb\times m_\ell\times m_0$.
\end{tabular}$$

Our next theorem is an improvement to the second last row of Table~\ref{table:original}.
\begin{theorem}\label{thm:p-5}
Let $p\geq 17$ be a prime. There exists a covering system of the integers such that all moduli are odd, greater than 1, and distinct except that the modulus $p$ is used $p-5$ times. Further, the covering system does not have $p^2$ as a modulus.
\end{theorem}
\begin{proof}
A tree diagram of such a covering system when $p=17$ is given by Figures~\ref{fig:p-5ps_second}-\ref{fig:p-5psT5_second}, with the main tree in Figure~\ref{fig:p-5ps_second} and subtrees in Figures\ref{fig:p-5psT1_second}-\ref{fig:p-5psT5_second}. Here, $q>31$ is a prime. Since no modulus in this covering system is divisible by $p^2$, our proof is completed by applying Lemma~5.1 of Harrington, Sun, and Wong \cite{hsw}.
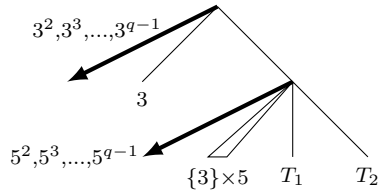
\begin{figure}[H]
\centering
\begin{tikzpicture}[scale=1]
\draw[ultra thick,-latex](0,0)--(-2,-1);\node[above]at(-1.6,-0.6){$\substack{3^2,3^3,\dotsc,3^{q-1}}$};
\draw(0,0)--(-1,-1);\node[below]at(-1,-1){$\substack{3}$};
\draw(0,0)--(1,-1);

\draw[ultra thick,-latex](1,-1)--(-1,-2);\node at(-1.9,-2){$\substack{5^2,5^3,\dotsc,5^{q-1}}$};
\draw(1,-1)--(-0.125,-2)--(0.125,-2)--(1,-1);\node[below]at(0,-2){$\substack{\{3\}\times5}$};

\draw(1,-1)--(1,-2);\node[below]at(1,-2){$\substack{T_1}$};
\draw(1,-1)--(2,-2);\node[below]at(2,-2){$\substack{T_2}$};
\end{tikzpicture}
\caption{An odd covering system with the modulus $p$ ($p\geq17$) used exactly $p-5$ times}
\label{fig:p-5ps_second}
\end{figure}

\begin{figure}[H]
\centering
\begin{tikzpicture}[scale=1]
\draw(0,0)--(-7.3,-1)--(-6.7,-1)--(0,0);\node[above]at(-5.5,-0.8){$\substack{p,p,\dotsc,p\\p-5\text{ branches}}$};
\draw(0,0)--(-6,-1);\node[below]at(-6.2,-1){$\substack{3\times p}$};

\draw(0,0)--(-5,-1);\draw[ultra thick,-latex](-5,-1)--(-7.1,-2);\node[above]at(-7.1,-2){$\substack{7^2,7^3,\\\dotsc,7^{q-1}}$};
\draw(-5,-1)--(-6.725,-2)--(-6.475,-2)--(-5,-1);\node[below]at(-6.7,-2){$\substack{\{3,5\}\\\times7}$};
\draw(-5,-1)--(-6,-2);\node[below]at(-6,-2){$\substack{5\times p\\\times7}$};

\draw(-5,-1)--(-4.5,-2);\draw[ultra thick,-latex](-4.5,-2)--(-6.5,-3);\node at(-7,-3.2){$\substack{11^2,11^3,\\\dotsc,11^{q-1}}$};
\draw(-4.5,-2)--(-6.125,-3)--(-5.875,-3)--(-4.5,-2);\node[below]at(-6,-3){$\substack{\{3,7\}\\\times5\\\times11}$};
\draw(-4.5,-2)--(-5.325,-3)--(-5.075,-3)--(-4.5,-2);\node[below]at(-5.2,-3){$\substack{\{3\}\\\times7\\\times11}$};
\draw(-4.5,-2)--(-4.525,-3)--(-4.275,-3)--(-4.5,-2);\node[below]at(-4.4,-3){$\substack{\{3\}\\\times11}$};
\draw(-4.5,-2)--(-3.725,-3)--(-3.475,-3)--(-4.5,-2);\node[below]at(-3.6,-3){$\substack{\{3\}\\\times5\\\times p\times11}$};

\draw(0,0)--(-2,-1);\draw[ultra thick,-latex](-2,-1)--(-4,-2);\node[above]at(-3.6,-1.6){$\substack{7^2,7^3,\dotsc,7^{q-1}}$};
\draw(-2,-1)--(-3.625,-2)--(-3.375,-2)--(-2,-1);\node[below]at(-3.5,-2){$\substack{\{3,5\}\\\times7}$};
\draw(-2,-1)--(-2.7,-2);\node[below]at(-2.7,-2){$\substack{3\times5\\\times p\times7}$};

\draw(-2,-1)--(-1.5,-2);\draw[ultra thick,-latex](-1.5,-2)--(-2.1,-3);\node at(-2.6,-3.2){$\substack{11^2,11^3,\\\dotsc,11^{q-1}}$};
\draw(-1.5,-2)--(-1.725,-3)--(-1.475,-3)--(-1.5,-2);\node[below]at(-1.6,-3){$\substack{\{3,7\}\\\times5\\\times11}$};
\draw(-1.5,-2)--(-0.925,-3)--(-0.675,-3)--(-1.5,-2);\node[below]at(-0.8,-3){$\substack{\{3\}\\\times7\\\times11}$};
\draw(-1.5,-2)--(-0.125,-3)--(0.125,-3)--(-1.5,-2);\node[below]at(0,-3){$\substack{\{3\}\\\times11}$};
\draw(-1.5,-2)--(0.675,-3)--(0.925,-3)--(-1.5,-2);\node[below]at(0.8,-3){$\substack{\{3\}\\\times5\times p\\\times7\times11}$};

\draw(0,0)--(1,-1);\draw[ultra thick,-latex](1,-1)--(-0.5,-2);\node[above]at(-0.4,-1.6){$\substack{7^2,7^3,\dotsc,7^{q-1}}$};
\draw(1,-1)--(0.275,-2)--(0.525,-2)--(1,-1);\node[below]at(0.4,-2){$\substack{\{3,5\}\\\times7}$};
\draw(1,-1)--(1.2,-2);\node[below]at(1.2,-2){$\substack{p\times7}$};

\draw(1,-1)--(3.3,-2);\draw[ultra thick,-latex](3.3,-2)--(1.4,-3);\node at(2.2,-2.1){$\substack{11^2,11^3,\\\dotsc,11^{q-1}}$};
\draw(3.3,-2)--(1.875,-3)--(2.125,-3)--(3.3,-2);\node[below]at(2,-3){$\substack{\{3,7\}\\\times5\\\times11}$};
\draw(3.3,-2)--(2.675,-3)--(2.925,-3)--(3.3,-2);\node[below]at(2.8,-3){$\substack{\{3\}\\\times7\\\times11}$};
\draw(3.3,-2)--(3.475,-3)--(3.725,-3)--(3.3,-2);\node[below]at(3.6,-3){$\substack{\{3\}\\\times11}$};
\draw(3.3,-2)--(4.275,-3)--(4.525,-3)--(3.3,-2);\node[below]at(4.4,-3){$\substack{\{3\}\\\times p\\\times7\times11}$};

\draw(0,0)--(4.9,-1);\draw[ultra thick,-latex](4.9,-1)--(3.8,-2);\node[above]at(3.6,-1.6){$\substack{7^2,7^3,\dotsc,7^{q-1}}$};
\draw(4.9,-1)--(4.275,-2)--(4.525,-2)--(4.9,-1);\node[below]at(4.4,-2){$\substack{\{3,5\}\\\times7}$};
\draw(4.9,-1)--(5.2,-2);\node[below]at(5.2,-2){$\substack{3\times p\\\times7}$};

\draw(4.9,-1)--(7.3,-2);\draw[ultra thick,-latex](7.3,-2)--(5.4,-3);\node at(6.2,-2.1){$\substack{11^2,11^3,\\\dotsc,11^{q-1}}$};
\draw(7.3,-2)--(5.875,-3)--(6.125,-3)--(7.3,-2);\node[below]at(6,-3){$\substack{\{3,7\}\\\times5\\\times11}$};
\draw(7.3,-2)--(6.675,-3)--(6.925,-3)--(7.3,-2);\node[below]at(6.8,-3){$\substack{\{3\}\\\times7\\\times11}$};
\draw(7.3,-2)--(7.475,-3)--(7.725,-3)--(7.3,-2);\node[below]at(7.6,-3){$\substack{\{3\}\\\times11}$};
\draw(7.3,-2)--(8.275,-3)--(8.525,-3)--(7.3,-2);\node[below]at(8.4,-3){$\substack{\{3\}\\\times p\\\times11}$};
\end{tikzpicture}
\caption{$T_1$ in Figure~\ref{fig:p-5ps_second}}
\label{fig:p-5psT1_second}
\end{figure}

\begin{figure}[H]
\centering
\begin{tikzpicture}[scale=1]
\draw(0,0)--(-7.3,-1)--(-6.7,-1)--(0,0);\node[above]at(-5.5,-0.8){$\substack{p,p,\dotsc,p\\p-5\text{ branches}}$};
\draw(0,0)--(-6,-1);\node[below]at(-6.2,-1){$\substack{3\times p}$};
\draw(0,0)--(-5.3,-1);\node[below]at(-5.4,-1){$\substack{5\times p}$};
\draw(0,0)--(-4.4,-1);\node[below]at(-4.4,-1){$\substack{3\times5\times p}$};

\draw(0,0)--(-2.2,-1);\draw[ultra thick,-latex](-2.2,-1)--(-6,-2);\node at(-6.5,-2){$\substack{7^2,7^3,\\\dotsc,7^{q-1}}$};
\draw(-2.2,-1)--(-5.5,-2)--(-5.1,-2)--(-2.2,-1);\node[below]at(-5.5,-2){$\substack{\{3\}\times7}$};

\draw(-2.2,-1)--(-4.2,-2);\draw[ultra thick,-latex](-4.2,-2)--(-6.2,-3);\node at(-6.8,-3){$\substack{13^2,13^3,\\\dotsc,13^{q-1}}$};
\draw(-4.2,-2)--(-5.825,-3)--(-5.575,-3)--(-4.2,-2);\node[below]at(-5.7,-3){$\substack{\{3,5,p\}\\\times13}$};
\draw(-4.2,-2)--(-4.825,-3)--(-4.575,-3)--(-4.2,-2);\node[below]at(-4.7,-3){$\substack{\{3,5\}\\\times7\times13}$};

\draw(-2.2,-1)--(-2.1,-2);\draw[ultra thick,-latex](-2.1,-2)--(-4.2,-3);\node at(-3.3,-2.1){$\substack{13^2,13^3,\\\dotsc,13^{q-1}}$};
\draw(-2.1,-2)--(-3.725,-3)--(-3.475,-3)--(-2.1,-2);\node[below]at(-3.6,-3){$\substack{\{3,5,p\}\\\times13}$};
\draw(-2.1,-2)--(-2.725,-3)--(-2.475,-3)--(-2.1,-2);\node[below]at(-2.6,-3){$\substack{\{3,5\}\\\times p\times7\\\times13}$};

\draw(-2.2,-1)--(-1.7,-2);\node[below]at(-1.7,-2){$\substack{p\times7}$};
\draw(-2.2,-1)--(-1,-2);\node[below]at(-1,-2){$\substack{T_3}$};

\draw(0,0)--(3.5,-1);\draw[ultra thick,-latex](3.5,-1)--(-0.5,-2);\node[above]at(1,-1.6){$\substack{7^2,7^3,\dotsc,7^{q-1}}$};
\draw(3.5,-1)--(0,-2)--(0.4,-2)--(3.5,-1);\node[below]at(0.1,-2){$\substack{\{3\}\times7}$};
\draw(3.5,-1)--(0.9,-2);\node[below]at(0.9,-2){$\substack{T_4}$};

\draw(3.5,-1)--(2,-2);\draw[ultra thick,-latex](2,-2)--(-0.1,-3);\node at(-0.7,-3){$\substack{19^2,19^3,\\\dotsc,19^{q-1}}$};
\draw(2,-2)--(0.475,-3)--(0.725,-3)--(2,-2);\node[below]at(0.6,-3){$\substack{\{3,5,p,7\}\\\times19}$};
\draw(2,-2)--(2,-3);\draw[ultra thick,-latex](2,-3)--(0.5,-4);\node at(-0.6,-4){$\substack{13^2,13^3,\dotsc,13^{q-1}}$};
\draw(2,-3)--(0.875,-4)--(1.125,-4)--(2,-3);\node[below]at(1,-4){$\substack{\{3,5\}\\\times13}$};
\draw(2,-3)--(1.875,-4)--(2.125,-4)--(2,-3);\node[below]at(2,-4){$\substack{\{3,5,p\}\\\times19\times13}$};
\draw(2,-2)--(4,-3);\draw[ultra thick,-latex](4,-3)--(2.5,-4);\node[above]at(2.9,-3.6){$\substack{13^2,13^3,\\\dotsc,13^{q-1}}$};
\draw(4,-3)--(3.075,-4)--(3.375,-4)--(4,-3);\node[below]at(3.2,-4){$\substack{\{3,5\}\\\times13}$};
\draw(4,-3)--(4.075,-4)--(4.325,-4)--(4,-3);\node[below]at(4.2,-4){$\substack{\{3,5,p\}\\\times7\times19\\\times13}$};

\draw(3.5,-1)--(3.5,-2);\node[below]at(3.5,-2){$\substack{3\times p\times7}$};
\draw(3.5,-1)--(4.5,-2);\node[below]at(4.5,-2){$\substack{T_5}$};
\end{tikzpicture}
\caption{$T_2$ in Figure~\ref{fig:p-5ps_second}}
\label{fig:p-5psT2_second}
\end{figure}

\begin{figure}[H]
\centering
\begin{tikzpicture}[scale=0.98]
\draw[ultra thick,-latex](0,0)--(-8,-1);\node[above]at(-8,-0.9){$\substack{11^2,11^3,\dotsc,11^{q-1}}$};
\draw(0,0)--(4.4,-1)--(4.8,-1)--(0,0);\node[below]at(4.7,-1){$\substack{\{3\}\\\times7\times11}$};
\draw(0,0)--(5.25,-1)--(5.75,-1)--(0,0);\node[below]at(5.7,-1){$\substack{\{3\}\times11}$};
\draw(0,0)--(6.25,-1)--(6.75,-1)--(0,0);\node[above]at(6.2,-0.9){$\substack{\{3\}\times p\times7\times11}$};
\draw(0,0)--(-6,-1)--(-7.625,-2)--(-7.375,-2)--(-6,-1)--(-6.425,-2)--(-6.175,-2)--(-6,-1);\draw[ultra thick,-latex](-6,-1)--(-8.3,-2);
\node[above]at(-8.1,-1.6){$\substack{13^2,13^3,\dotsc,13^{q-1}}$};\node[below]at(-7.5,-2){$\substack{\{3,5,p\}\\\times13}$};\node[below]at(-6.3,-2){$\substack{\{3,5\}\\\times11\times13}$};

\draw(0,0)--(-2.8,-1)--(-4.425,-2)--(-4.175,-2)--(-2.8,-1)--(-3.025,-2)--(-2.775,-2)--(-2.8,-1);\draw[ultra thick,-latex](-2.8,-1)--(-5.2,-2);
\node[above]at(-4.8,-1.6){$\substack{13^2,13^3,\dotsc,13^{q-1}}$};\node[below]at(-4.3,-2){$\substack{\{3,5,p\}\\\times13}$};\node[below]at(-2.9,-2){$\substack{\{3,5\}\\\times p\times11\times13}$};

\draw(0,0)--(0.5,-1)--(-1.125,-2)--(-0.875,-2)--(0.5,-1)--(0.275,-2)--(0.525,-2)--(0.5,-1);\draw[ultra thick,-latex](0.5,-1)--(-1.8,-2);
\node[above]at(-1.5,-1.6){$\substack{13^2,13^3,\dotsc,13^{q-1}}$};\node[below]at(-1,-2){$\substack{\{3,5,p\}\\\times13}$};\node[below]at(0.4,-2){$\substack{\{3,5\}\\\times7\times11\times13}$};

\draw(0,0)--(3.6,-1)--(2.275,-2)--(2.525,-2)--(3.6,-1)--(3.875,-2)--(4.125,-2)--(3.6,-1);\draw[ultra thick,-latex](3.6,-1)--(1.8,-2);
\node[above]at(1.8,-1.6){$\substack{13^2,13^3,\dotsc,13^{q-1}}$};\node[below]at(2.4,-2){$\substack{\{3,5,p\}\\\times13}$};\node[below]at(4,-2){$\substack{\{3,5\}\\\times p\times7\times11\times13}$};
\end{tikzpicture}
\caption{$T_3$ in Figure~\ref{fig:p-5psT2_second}}
\label{fig:p-5psT3_second}
\end{figure}

\begin{figure}[H]
\centering
\begin{tikzpicture}[scale=0.95]
\draw[ultra thick,-latex](0,0)--(-6,-1);\node[above]at(-4.6,-0.7){$\substack{23^2,23^3,\dotsc,23^{q-1}}$};
\draw(0,0)--(-5.2,-1)--(-4.8,-1)--(0,0);\node[below]at(-5.2,-1){$\substack{\{3,5,p,7\}\times23}$};

\draw(0,0)--(-3.5,-1);\draw[ultra thick,-latex](-3.5,-1)--(-5,-2);\node at(-5.6,-2){$\substack{11^2,11^3,\\\dotsc,11^{q-1}}$};
\draw(-3.5,-1)--(-4.525,-2)--(-4.275,-2)--(-3.5,-1);\node[below]at(-4.4,-2){$\substack{\{3,5,p\}\\\times23\times11}$};
\draw(-3.5,-1)--(-3.425,-2)--(-3.175,-2)--(-3.5,-1);\node[below]at(-3.3,-2){$\substack{\{3,p\}\\\times11}$};

\draw(0,0)--(-1.6,-1);\draw[ultra thick,-latex](-1.6,-1)--(-2.7,-2);\node[above]at(-2.7,-1.7){$\substack{11^2,11^3,\\\dotsc,11^{q-1}}$};
\draw(-1.6,-1)--(-2.225,-2)--(-1.975,-2)--(-1.6,-1);\node[below]at(-2.1,-2){$\substack{\{3,5,p\}\\\times7\times23\\\times11}$};
\draw(-1.6,-1)--(-1.125,-2)--(-0.875,-2)--(-1.6,-1);\node[below]at(-1,-2){$\substack{\{3,p\}\\\times11}$};

\draw(0,0)--(0.3,-1);\draw[ultra thick,-latex](0.3,-1)--(-0.5,-2);\node[above]at(-0.6,-1.6){$\substack{13^2,13^3,\\\dotsc,13^{q-1}}$};
\draw(0.3,-1)--(-0.125,-2)--(0.125,-2)--(0.3,-1);\node[below]at(0,-2){$\substack{\{3,5\}\\\times13}$};
\draw(0.3,-1)--(0.675,-2)--(0.925,-2)--(0.3,-1);\node[below]at(0.8,-2){$\substack{\{3,5\}\\\times23\\\times13}$};
\draw(0.3,-1)--(1.475,-2)--(1.725,-2)--(0.3,-1);\node[below]at(1.6,-2){$\substack{\{3,5\}\\\times7\\\times13}$};

\draw(0,0)--(2.5,-1);\draw[ultra thick,-latex](2.5,-1)--(2.1,-2);\node[above]at(1.7,-1.7){$\substack{13^2,13^3,\\\dotsc,13^{q-1}}$};
\draw(2.5,-1)--(2.575,-2)--(2.825,-2)--(2.5,-1);\node[below]at(2.7,-2){$\substack{\{3,5\}\\\times13}$};
\draw(2.5,-1)--(3.375,-2)--(3.625,-2)--(2.5,-1);\node[below]at(3.5,-2){$\substack{\{3,5\}\\\times p\\\times23\\\times13}$};
\draw(2.5,-1)--(4.175,-2)--(4.425,-2)--(2.5,-1);\node[below]at(4.3,-2){$\substack{\{3,5\}\\\times7\\\times13}$};

\draw(0,0)--(5.5,-1);\draw[ultra thick,-latex](5.5,-1)--(4.8,-2);\node[above]at(4.5,-1.8){$\substack{13^2,13^3,\\\dotsc,13^{q-1}}$};
\draw(5.5,-1)--(5.175,-2)--(5.425,-2)--(5.5,-1);\node[below]at(5.3,-2){$\substack{\{3,5\}\\\times13}$};
\draw(5.5,-1)--(5.975,-2)--(6.225,-2)--(5.5,-1);\node[below]at(6.1,-2){$\substack{\{3,5\}\\\times7\\\times23\\\times13}$};
\draw(5.5,-1)--(6.775,-2)--(7.025,-2)--(5.5,-1);\node[below]at(6.9,-2){$\substack{\{3,5\}\\\times7\\\times13}$};

\draw(0,0)--(8.5,-1);\draw[ultra thick,-latex](8.5,-1)--(7.3,-2);\node[above]at(7.2,-1.8){$\substack{13^2,13^3,\\\dotsc,13^{q-1}}$};
\draw(8.5,-1)--(7.675,-2)--(7.925,-2)--(8.5,-1);\node[below]at(7.8,-2){$\substack{\{3,5\}\\\times13}$};
\draw(8.5,-1)--(8.575,-2)--(8.825,-2)--(8.5,-1);\node[below]at(8.7,-2){$\substack{\{3,5\}\\\times p\times7\\\times23\\\times13}$};
\draw(8.5,-1)--(9.375,-2)--(9.625,-2)--(8.5,-1);\node[below]at(9.5,-2){$\substack{\{3,5\}\\\times7\\\times13}$};
\end{tikzpicture}
\caption{$T_4$ in Figure~\ref{fig:p-5psT2_second}}
\label{fig:p-5psT4_second}
\end{figure}

\begin{figure}[H]
\centering
\begin{tikzpicture}[scale=0.95]
\draw[ultra thick,-latex](0,0)--(-8.3,-1);\node[above]at(-6,-0.7){$\substack{11^2,11^3,\dotsc,11^{q-1}}$};
\draw(0,0)--(5.25,-1)--(5.75,-1)--(0,0);\node[below]at(5.7,-1){$\substack{\{3\}\times7\times11}$};
\draw(0,0)--(6.25,-1)--(6.75,-1)--(0,0);\node[above]at(6.2,-0.9){$\substack{\{3,p\}\times11}$};
\draw(0,0)--(-7,-1)--(-8.325,-2)--(-8.075,-2)--(-7,-1);\draw[ultra thick,-latex](-7,-1)--(-8.8,-2);
\node[above]at(-8.8,-1.9){$\substack{29^2,29^3,\\\dotsc,29^{q-1}}$};\node[below]at(-8.2,-2){$\substack{\{3,5,p,7,11\}\\\times29\\28\text{ branches}}$};

\draw(0,0)--(-5,-1)--(-5.925,-2)--(-5.675,-2)--(-5,-1);\draw[ultra thick,-latex](-5,-1)--(-6.4,-2);
\node[above]at(-6.6,-2){$\substack{31^2,31^3,\\\dotsc,31^{q-1}}$};\node[below]at(-5.8,-2){$\substack{\{3,5,p,7,11\}\\\times31\\30\text{ branches}}$};

\draw(0,0)--(-3,-1)--(-4.225,-2)--(-3.975,-2)--(-3,-1)--(-3.125,-2)--(-2.875,-2)--(-3,-1);\draw[ultra thick,-latex](-3,-1)--(-4.6,-2);
\node[above]at(-4.4,-1.8){$\substack{19^2,19^3,\\\dotsc,19^{q-1}}$};\node[below]at(-4.1,-2){$\substack{\{3,5,p\}\\\times19}$};\node[below]at(-3,-2){$\substack{\{3,5,p\}\\\times11\times19}$};

\draw(-3,-1)--(-1.5,-2);
\draw[ultra thick,-latex](-1.5,-2)--(-2.5,-3);\node at(-3.6,-3){$\substack{13^2,13^3,\dotsc,13^{q-1}}$};
\draw(-1.5,-2)--(-2.125,-3)--(-1.875,-3)--(-1.5,-2);\node[below]at(-2,-3){$\substack{\{3,5\}\\\times13}$};
\draw(-1.5,-2)--(-1.225,-3)--(-0.975,-3)--(-1.5,-2);\node[below]at(-1,-3){$\substack{\{3,5,p\}\\\times19\\\times13}$};

\draw(-3,-1)--(0.2,-2);
\draw[ultra thick,-latex](0.2,-2)--(0,-3);\node[above]at(-0.5,-2.8){$\substack{13^2,13^3,\\\dotsc,13^{q-1}}$};
\draw(0.2,-2)--(0.275,-3)--(0.525,-3)--(0.2,-2);\node[below]at(0.4,-3){$\substack{\{3,5\}\\\times13}$};
\draw(0.2,-2)--(1.175,-3)--(1.425,-3)--(0.2,-2);\node[below]at(1.4,-3){$\substack{\{3,5,p\}\\\times11\\\times19\\\times13}$};

\draw(0,0)--(3,-1)--(1.275,-2)--(1.525,-2)--(3,-1)--(2.375,-2)--(2.625,-2)--(3,-1);\draw[ultra thick,-latex](3,-1)--(0.8,-2);
\node[above]at(1.1,-1.6){$\substack{19^2,19^3,\dotsc,19^{q-1}}$};\node[below]at(1.4,-2){$\substack{\{3,5,p\}\\\times19}$};\node[below]at(2.5,-2){$\substack{\{3,5,p\}\\\times7\times11\\\times19}$};

\draw(3,-1)--(4.5,-2);\draw[ultra thick,-latex](4.5,-2)--(3.1,-3);\node[above]at(3.6,-2.5){$\substack{13^2,13^3,\\\dotsc,13^{q-1}}$};
\draw(4.5,-2)--(3.475,-3)--(3.725,-3)--(4.5,-2);\node[below]at(3.6,-3){$\substack{\{3,5\}\\\times13}$};
\draw(4.5,-2)--(4.475,-3)--(4.725,-3)--(4.5,-2);\node[below]at(4.6,-3){$\substack{\{3,5,p\}\\\times19\\\times13}$};

\draw(3,-1)--(6.5,-2);\draw[ultra thick,-latex](6.5,-2)--(5.1,-3);\node[above]at(5.3,-2.7){$\substack{13^2,13^3,\\\dotsc,13^{q-1}}$};
\draw(6.5,-2)--(5.475,-3)--(5.725,-3)--(6.5,-2);\node[below]at(5.6,-3){$\substack{\{3,5\}\\\times13}$};
\draw(6.5,-2)--(6.475,-3)--(6.725,-3)--(6.5,-2);\node[below]at(6.6,-3){$\substack{\{3,5,p\}\\\times7\times11\\\times19\\\times13}$};
\end{tikzpicture}
\caption{$T_5$ in Figure~\ref{fig:p-5psT2_second}}
\label{fig:p-5psT5_second}
\end{figure}
\end{proof}

Our last four theorems of this section establish $t_9\leq3$, $t_{15}\leq4$, $t_{21}\leq5$, and $t_{25}\leq8$, respectively.
\begin{theorem}\label{thm:three9s}
There exists a covering system of the integers such that all moduli are odd, greater than 1, and distinct except that the modulus $9$ is used three times. 
\end{theorem}
\begin{proof}
A tree diagram of such a covering system is given by Figures~\ref{fig:three9s}-\ref{fig:three9sT2}, with the main tree in Figure~\ref{fig:three9s} and subtrees in Figures~\ref{fig:three9sT1} and \ref{fig:three9sT2}. Here, $q>29$ is a prime.

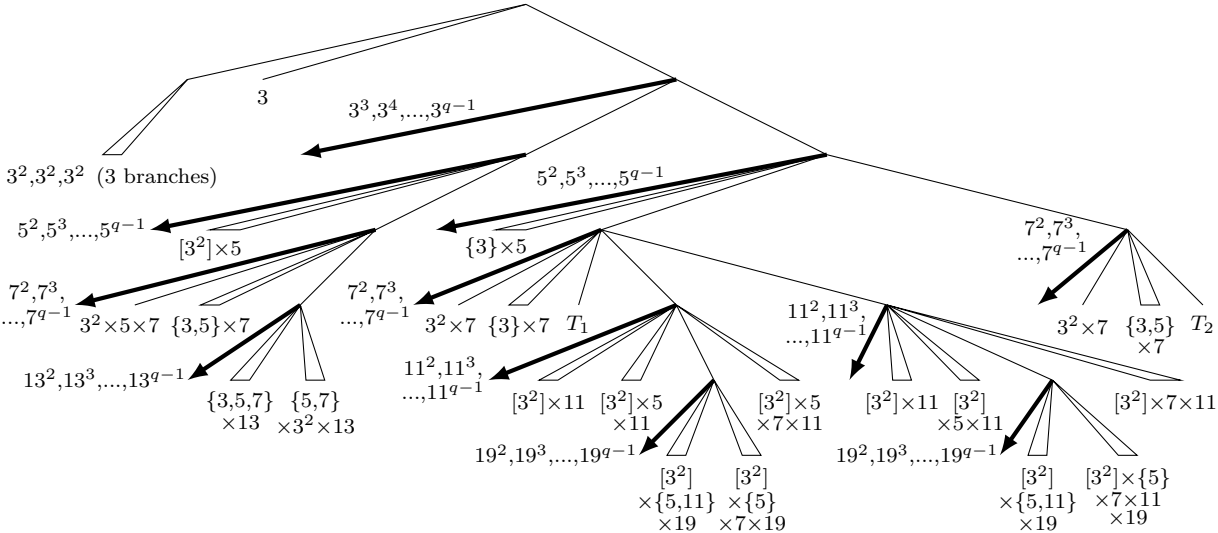
\begin{figure}[H]
\centering
\begin{tikzpicture}[scale=1]
\draw(0,0)--(-4.5,-1)--(-5.625,-2)--(-5.375,-2)--(-4.5,-1);\node[below]at(-5.5,-2){$\substack{3^2,3^2,3^2\ (3\text{ branches})}$};
\draw(0,0)--(-3.5,-1);\node[below]at(-3.5,-1){$\substack{3}$};

\draw(0,0)--(2,-1);
\draw[ultra thick,-latex](2,-1)--(-3,-2);\node[above]at(-1.5,-1.7){$\substack{3^3,3^4,\dotsc,3^{q-1}}$};

\draw(2,-1)--(0,-2);
\draw[ultra thick,-latex](0,-2)--(-5,-3);\node at(-5.9,-3){$\substack{5^2,5^3,\dotsc,5^{q-1}}$};
\draw(0,-2)--(-4.2,-3)--(-3.8,-3)--(0,-2);\node[below]at(-4.2,-2.95){$\substack{[3^2]\times5}$};

\draw(0,-2)--(-2,-3);
\draw[ultra thick,-latex](-2,-3)--(-6,-4);\node at(-6.5,-4){$\substack{7^2,7^3,\\\dotsc,7^{q-1}}$};
\draw(-2,-3)--(-5.2,-4);\node[below]at(-5.4,-3.98){$\substack{3^2\times5\times7}$};
\draw(-2,-3)--(-4.325,-4)--(-4.075,-4)--(-2,-3);\node[below]at(-4.2,-4){$\substack{\{3,5\}\times7}$};

\draw(-2,-3)--(-3,-4);
\draw[ultra thick,-latex](-3,-4)--(-4.5,-5);\node at(-5.6,-5){$\substack{13^2,13^3,\dotsc,13^{q-1}}$};
\draw(-3,-4)--(-3.925,-5)--(-3.675,-5)--(-3,-4);\node[below]at(-3.8,-5){$\substack{\{3,5,7\}\\\times13}$};
\draw(-3,-4)--(-2.925,-5)--(-2.675,-5)--(-3,-4);\node[below]at(-2.8,-5){$\substack{\{5,7\}\\\times3^2\times13}$};

\draw(2,-1)--(4,-2);
\draw[ultra thick,-latex](4,-2)--(-1.2,-3);\node[above]at(1,-2.6){$\substack{5^2,5^3,\dotsc,5^{q-1}}$};
\draw(4,-2)--(-0.4,-3)--(0,-3)--(4,-2);\node[below]at(-0.4,-2.95){$\substack{\{3\}\times5}$};

\draw(4,-2)--(1,-3);
\draw[ultra thick,-latex](1,-3)--(-1.5,-4);\node at(-2,-4){$\substack{7^2,7^3,\\\dotsc,7^{q-1}}$};
\draw(1,-3)--(-0.9,-4);\node[below]at(-1,-3.98){$\substack{3^2\times7}$};
\draw(1,-3)--(-0.225,-4)--(0.025,-4)--(1,-3);\node[below]at(-0.1,-4){$\substack{\{3\}\times7}$};

\draw(1,-3)--(0.7,-4);\node[below]at(0.7,-4){$\substack{T_1}$};
\draw(1,-3)--(2,-4);
\draw[ultra thick,-latex](2,-4)--(-0.5,-5);\node at(-1.1,-5){$\substack{11^2,11^3,\\\dotsc,11^{q-1}}$};
\draw(2,-4)--(0.175,-5)--(0.425,-5)--(2,-4);\node[below]at(0.3,-5){$\substack{[3^2]\times11}$};
\draw(2,-4)--(1.275,-5)--(1.525,-5)--(2,-4);\node[below]at(1.4,-5){$\substack{[3^2]\times5\\\times11}$};

\draw(2,-4)--(2.5,-5);
\draw[ultra thick,-latex](2.5,-5)--(1.5,-6);\node at(0.4,-6){$\substack{19^2,19^3,\dotsc,19^{q-1}}$};
\draw(2.5,-5)--(1.875,-6)--(2.125,-6)--(2.5,-5);\node[below]at(2,-6){$\substack{[3^2]\\\times\{5,11\}\\\times19}$};
\draw(2.5,-5)--(2.875,-6)--(3.125,-6)--(2.5,-5);\node[below]at(3,-6){$\substack{[3^2]\\\times\{5\}\\\times7\times19}$};

\draw(2,-4)--(3.375,-5)--(3.625,-5)--(2,-4);\node[below]at(3.5,-5){$\substack{[3^2]\times5\\\times7\times11}$};

\draw(1,-3)--(4.8,-4);
\draw[ultra thick,-latex](4.8,-4)--(4.3,-5);\node[above] at(4,-4.7){$\substack{11^2,11^3,\\\dotsc,11^{q-1}}$};
\draw(4.8,-4)--(4.875,-5)--(5.125,-5)--(4.8,-4);\node[below]at(5,-5){$\substack{[3^2]\times11}$};
\draw(4.8,-4)--(5.775,-5)--(6.025,-5)--(4.8,-4);\node[below]at(5.9,-5){$\substack{[3^2]\\\times5\times11}$};

\draw(4.8,-4)--(7,-5);
\draw[ultra thick,-latex](7,-5)--(6.3,-6);\node at(5.2,-6){$\substack{19^2,19^3,\dotsc,19^{q-1}}$};
\draw(7,-5)--(6.675,-6)--(6.925,-6)--(7,-5);\node[below]at(6.8,-6){$\substack{[3^2]\\\times\{5,11\}\\\times19}$};
\draw(7,-5)--(7.875,-6)--(8.125,-6)--(7,-5);\node[below]at(8,-6){$\substack{[3^2]\times\{5\}\\\times7\times11\\\times19}$};

\draw(4.8,-4)--(8.3,-5)--(8.7,-5)--(4.8,-4);\node[below]at(8.5,-5){$\substack{[3^2]\times7\times11}$};

\draw(4,-2)--(8,-3);
\draw[ultra thick,-latex](8,-3)--(6.8,-4);\node[above]at(7,-3.6){$\substack{7^2,7^3,\\\dotsc,7^{q-1}}$};
\draw(8,-3)--(7.4,-4);\node[below]at(7.4,-4){$\substack{3^2\times7}$};
\draw(8,-3)--(8.175,-4)--(8.425,-4)--(8,-3);\node[below]at(8.3,-4){$\substack{\{3,5\}\\\times7}$};
\draw(8,-3)--(9,-4);\node[below]at(9,-4){$\substack{T_2}$};
\end{tikzpicture}
\caption{An odd covering with $9$ used exactly three times as a modulus}
\label{fig:three9s}
\end{figure}

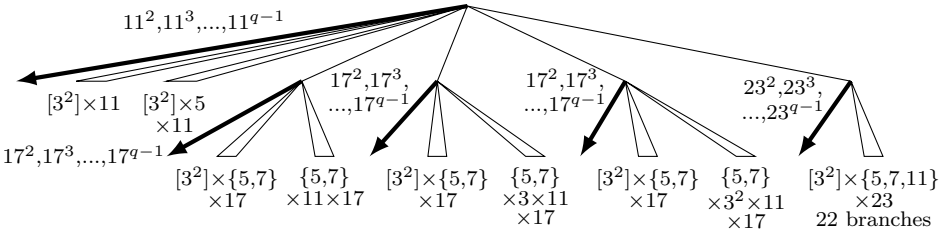
\begin{figure}[H]
\centering
\begin{tikzpicture}[scale=1]
\draw[ultra thick,-latex](0,0)--(-6,-1);\node[above]at(-3.5,-0.5){$\substack{11^2,11^3,\dotsc,11^{q-1}}$};
\draw(0,0)--(-5.2,-1)--(-4.8,-1)--(0,0);\node[below]at(-5.1,-1){$\substack{[3^2]\times11}$};
\draw(0,0)--(-4,-1)--(-3.6,-1)--(0,0);\node[below]at(-3.9,-1){$\substack{[3^2]\times5\\\times11}$};

\draw(0,0)--(-2.2,-1);
\draw[ultra thick,-latex](-2.2,-1)--(-4,-2);\node at(-5.1,-2){$\substack{17^2,17^3,\dotsc,17^{q-1}}$};
\draw(-2.2,-1)--(-3.325,-2)--(-3.075,-2)--(-2.2,-1);\node[below]at(-3.2,-2){$\substack{[3^2]\times\{5,7\}\\\times17}$};
\draw(-2.2,-1)--(-2.025,-2)--(-1.775,-2)--(-2.2,-1);\node[below]at(-1.9,-2){$\substack{\{5,7\}\\\times11\times17}$};

\draw(0,0)--(-0.4,-1);
\draw[ultra thick,-latex](-0.4,-1)--(-1.3,-2);\node[above]at(-1.3,-1.6){$\substack{17^2,17^3,\\\dotsc,17^{q-1}}$};
\draw(-0.4,-1)--(-0.525,-2)--(-0.275,-2)--(-0.4,-1);\node[below]at(-0.4,-2){$\substack{[3^2]\times\{5,7\}\\\times17}$};
\draw(-0.4,-1)--(0.775,-2)--(1.025,-2)--(-0.4,-1);\node[below]at(0.9,-2){$\substack{\{5,7\}\\\times3\times11\\\times17}$};

\draw(0,0)--(2.1,-1);
\draw[ultra thick,-latex](2.1,-1)--(1.5,-2);\node[above]at(1.3,-1.6){$\substack{17^2,17^3,\\\dotsc,17^{q-1}}$};
\draw(2.1,-1)--(2.275,-2)--(2.525,-2)--(2.1,-1);\node[below]at(2.4,-2){$\substack{[3^2]\times\{5,7\}\\\times17}$};
\draw(2.1,-1)--(3.575,-2)--(3.825,-2)--(2.1,-1);\node[below]at(3.7,-2){$\substack{\{5,7\}\\\times3^2\times11\\\times17}$};

\draw(0,0)--(5.1,-1);
\draw[ultra thick,-latex](5.1,-1)--(4.4,-2);\node[above]at(4.2,-1.7){$\substack{23^2,23^3,\\\dotsc,23^{q-1}}$};
\draw(5.1,-1)--(5.275,-2)--(5.525,-2)--(5.1,-1);\node[below]at(5.4,-2){$\substack{[3^2]\times\{5,7,11\}\\\times23\\22\text{ branches}}$};
\end{tikzpicture}
\caption{$T_1$ in Figure~\ref{fig:three9s}}
\label{fig:three9sT1}
\end{figure}

\begin{figure}[H]
\centering
\begin{tikzpicture}[scale=1]
\draw[ultra thick,-latex](0,0)--(-7.5,-1);\node[above]at(-4,-0.5){$\substack{11^2,11^3,\dotsc,11^{q-1}}$};
\draw(0,0)--(-6.6,-1)--(-6.2,-1)--(0,0);\node[below]at(-6.6,-0.9){$\substack{[3^2]\times11}$};

\draw(0,0)--(-5,-1);
\draw[ultra thick,-latex](-5,-1)--(-8,-2);\node at(-8.6,-2){$\substack{13^2,13^3,\\\dotsc,13^{q-1}}$};
\draw(-5,-1)--(-7.525,-2)--(-7.275,-2)--(-5,-1);\node[below]at(-7.5,-2){$\substack{\{3,5,7\}\\\times13}$};
\draw(-5,-1)--(-6.625,-2)--(-6.375,-2)--(-5,-1);\node[below]at(-6.5,-2){$\substack{\{5,7\}\\\times11\times13}$};

\draw(0,0)--(-2.5,-1);
\draw[ultra thick,-latex](-2.5,-1)--(-5.5,-2);\node[above]at(-4.2,-1.65){$\substack{13^2,13^3,\\\dotsc,13^{q-1}}$};
\draw(-2.5,-1)--(-4.925,-2)--(-4.675,-2)--(-2.5,-1);\node[below]at(-4.9,-2){$\substack{\{3,5,7\}\\\times13}$};
\draw(-2.5,-1)--(-3.925,-2)--(-3.675,-2)--(-2.5,-1);\node[below]at(-3.9,-2){$\substack{\{5,7\}\\\times3\times11\\\times13}$};

\draw(0,0)--(-0.3,-1);
\draw[ultra thick,-latex](-0.3,-1)--(-3.1,-2);\node[above]at(-1.8,-1.6){$\substack{13^2,13^3,\\\dotsc,13^{q-1}}$};
\draw(-0.3,-1)--(-2.625,-2)--(-2.375,-2)--(-0.3,-1);\node[below]at(-2.5,-2){$\substack{\{3,5,7\}\\\times13}$};
\draw(-0.3,-1)--(-1.625,-2)--(-1.375,-2)--(-0.3,-1);\node[below]at(-1.5,-2){$\substack{\{5,7\}\\\times3^2\times11\\\times13}$};

\draw(0,0)--(5,-1);
\draw[ultra thick,-latex](5,-1)--(-0.8,-2);\node[above]at(1.4,-1.6){$\substack{13^2,13^3,\dotsc,13^{q-1}}$};
\draw(5,-1)--(-0.2,-2)--(0.2,-2)--(5,-1);\node[below]at(0,-2){$\substack{\{3,5,7\}\times13}$};

\draw(5,-1)--(2,-2);
\draw[ultra thick,-latex](2,-2)--(-4.2,-4);\node[above]at(-3.7,-3.7){$\substack{29^2,29^3,\dotsc,29^{q-1}}$};
\draw(2,-2)--(-3.725,-4)--(-3.475,-4)--(2,-2);\node[below]at(-3.6,-4){$\substack{[3^2]\\\times\{5,7,11\}\\\times29}$};
\draw(2,-2)--(-2.425,-4)--(-2.175,-4)--(2,-2);\node[below]at(-2.3,-4){$\substack{[3^2]\times\{5\}\\\times13\times29}$};

\draw(5,-1)--(2.6,-2);
\draw[ultra thick,-latex](2.6,-2)--(-0.4,-4);\node at(-0.95,-4){$\substack{29^2,29^3,\\\dotsc,29^{q-1}}$};
\draw(2.6,-2)--(-0.025,-4)--(0.225,-4)--(2.6,-2);\node[below]at(0.1,-4){$\substack{[3^2]\\\times\{5,7,11\}\\\times29}$};
\draw(2.6,-2)--(1.275,-4)--(1.525,-4)--(2.6,-2);\node[below]at(1.4,-4){$\substack{[3^2]\times\{5\}\\\times7\times13\\\times29}$};

\draw(5,-1)--(4.5,-2);
\draw[ultra thick,-latex](4.5,-2)--(2.3,-4);\node[above]at(3.1,-3){$\substack{29^2,29^3,\\\dotsc,29^{q-1}}$};
\draw(4.5,-2)--(2.675,-4)--(2.925,-4)--(4.5,-2);\node[below]at(2.8,-4){$\substack{[3^2]\\\times\{5,7,11\}\\\times29}$};
\draw(4.5,-2)--(3.975,-4)--(4.225,-4)--(4.5,-2);\node[below]at(4.1,-4){$\substack{[3^2]\times\{5\}\\\times11\times13\\\times29}$};

\draw(5,-1)--(6,-2);
\draw[ultra thick,-latex](6,-2)--(4.9,-4);\node[above]at(5.1,-2.8){$\substack{29^2,29^3,\\\dotsc,29^{q-1}}$};
\draw(6,-2)--(5.275,-4)--(5.525,-4)--(6,-2);\node[below]at(5.4,-4){$\substack{[3^2]\\\times\{5,7,11\}\\\times29}$};
\draw(6,-2)--(6.575,-4)--(6.825,-4)--(6,-2);\node[below]at(6.7,-4){$\substack{[3^2]\times\{5\}\\\times7\times11\\\times13\times29}$};

\draw(0,0)--(6,-1)--(6.4,-1)--(0,0);\node[below]at(6.2,-1){$\substack{[3^2]\times7\times11}$};
\end{tikzpicture}
\caption{$T_2$ in Figure~\ref{fig:three9s}}
\label{fig:three9sT2}
\end{figure}
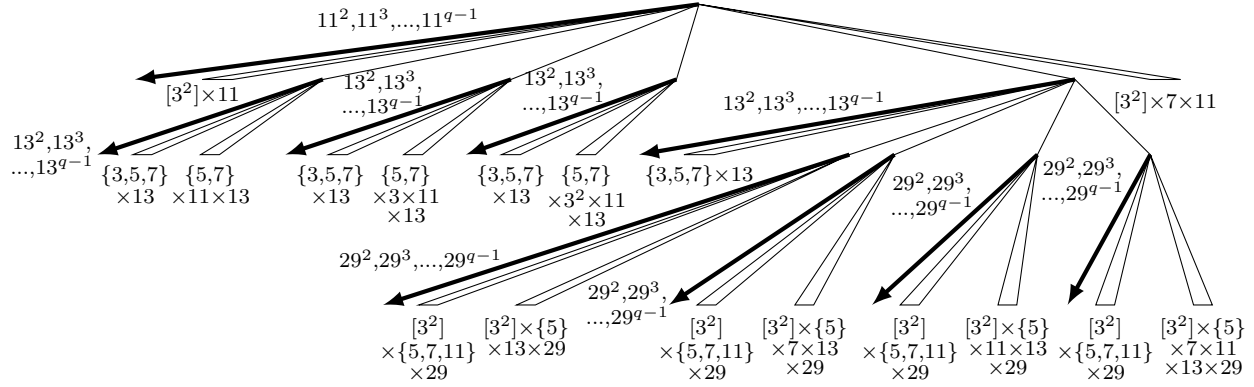
\end{proof}

\begin{theorem}\label{thm:four15s}
There exists a covering system of the integers such that all moduli are odd, greater than 1, and distinct except that the modulus $15$ is used four times.
\end{theorem}
\begin{proof}
A tree diagram of such a covering system is given by Figure~\ref{fig:four15s}. Here, $q>13$ is a prime.
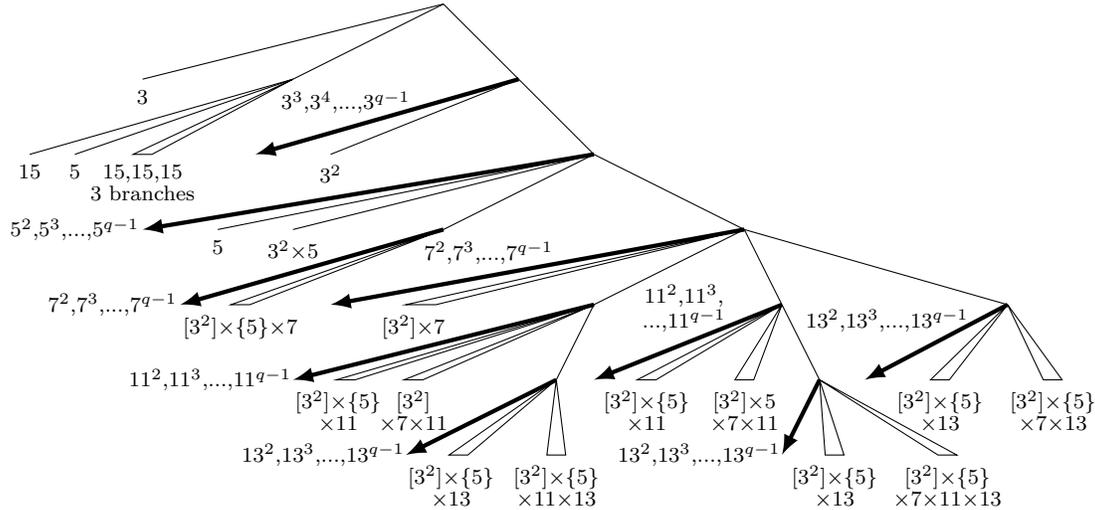
\begin{figure}[H]
\centering
\begin{tikzpicture}[scale=1]
\draw(0,0)--(-4,-1);\node[below]at(-4,-1){$\substack{3}$};
\draw(0,0)--(-2,-1);\draw(-2,-1)--(-5.5,-2);\node[below]at(-5.5,-2){$\substack{15}$};
\draw(-2,-1)--(-4.9,-2);\node[below]at(-4.9,-2){$\substack{5}$};
\draw(-2,-1)--(-4.125,-2)--(-3.875,-2)--(-2,-1);\node[below]at(-4,-2){$\substack{15,15,15\\3\text{ branches}}$};
\draw(0,0)--(1,-1);\draw[ultra thick,-latex](1,-1)--(-2.5,-2);\node[above]at(-1.3,-1.6){$\substack{3^3,3^4,\dotsc,3^{q-1}}$};
\draw(1,-1)--(-1.5,-2);\node[below]at(-1.5,-2){$\substack{3^2}$};

\draw(1,-1)--(2,-2);\draw[ultra thick,-latex](2,-2)--(-4,-3);\node at(-4.9,-3){$\substack{5^2,5^3,\dotsc,5^{q-1}}$};
\draw(2,-2)--(-3,-3);\node[below]at(-3,-3){$\substack{5}$};
\draw(2,-2)--(-2,-3);\node[below]at(-2,-3){$\substack{3^2\times5}$};

\draw(2,-2)--(0,-3);\draw[ultra thick,-latex](0,-3)--(-3.5,-4);\node at(-4.4,-4){$\substack{7^2,7^3,\dotsc,7^{q-1}}$};
\draw(0,-3)--(-2.825,-4)--(-2.575,-4)--(0,-3);\node[below]at(-2.7,-4){$\substack{[3^2]\times\{5\}\times7}$};

\draw(2,-2)--(4,-3);\draw[ultra thick,-latex](4,-3)--(-1.5,-4);\node[above]at(0.6,-3.6){$\substack{7^2,7^3,\dotsc,7^{q-1}}$};
\draw(4,-3)--(-0.5,-4)--(-0.1,-4)--(4,-3);\node[below]at(-0.4,-4){$\substack{[3^2]\times7}$};

\draw(4,-3)--(2,-4);\draw[ultra thick,-latex](2,-4)--(-2,-5);\node at(-3.1,-5){$\substack{11^2,11^3,\dotsc,11^{q-1}}$};
\draw(2,-4)--(-1.425,-5)--(-1.175,-5)--(2,-4);\node[below]at(-1.4,-5){$\substack{[3^2]\times\{5\}\\\times11}$};
\draw(2,-4)--(-0.525,-5)--(-0.275,-5)--(2,-4);\node[below]at(-0.4,-5){$\substack{[3^2]\\\times7\times11}$};

\draw(2,-4)--(1.5,-5);\draw[ultra thick,-latex](1.5,-5)--(-0.5,-6);\node at(-1.6,-6){$\substack{13^2,13^3,\dotsc,13^{q-1}}$};
\draw(1.5,-5)--(0.075,-6)--(0.325,-6)--(1.5,-5);\node[below]at(0.1,-6){$\substack{[3^2]\times\{5\}\\\times13}$};
\draw(1.5,-5)--(1.375,-6)--(1.625,-6)--(1.5,-5);\node[below]at(1.5,-6){$\substack{[3^2]\times\{5\}\\\times11\times13}$};

\draw(4,-3)--(4.5,-4);\draw[ultra thick,-latex](4.5,-4)--(2,-5);\node[above]at(3.2,-4.5){$\substack{11^2,11^3,\\\dotsc,11^{q-1}}$};
\draw(4.5,-4)--(2.575,-5)--(2.825,-5)--(4.5,-4);\node[below]at(2.7,-5){$\substack{[3^2]\times\{5\}\\\times11}$};
\draw(4.5,-4)--(3.875,-5)--(4.125,-5)--(4.5,-4);\node[below]at(4,-5){$\substack{[3^2]\times5\\\times7\times11}$};

\draw(4.5,-4)--(5,-5);\draw[ultra thick,-latex](5,-5)--(4.5,-6);\node at(3.4,-6){$\substack{13^2,13^3,\dotsc,13^{q-1}}$};
\draw(5,-5)--(5.075,-6)--(5.325,-6)--(5,-5);\node[below]at(5.2,-6){$\substack{[3^2]\times\{5\}\\\times13}$};
\draw(5,-5)--(6.575,-6)--(6.825,-6)--(5,-5);\node[below]at(6.7,-6){$\substack{[3^2]\times\{5\}\\\times7\times11\times13}$};

\draw(4,-3)--(7.5,-4);\draw[ultra thick,-latex](7.5,-4)--(5.6,-5);\node[above]at(5.9,-4.5){$\substack{13^2,13^3,\dotsc,13^{q-1}}$};
\draw(7.5,-4)--(6.475,-5)--(6.725,-5)--(7.5,-4);\node[below]at(6.6,-5){$\substack{[3^2]\times\{5\}\\\times{13}}$};
\draw(7.5,-4)--(7.975,-5)--(8.225,-5)--(7.5,-4);\node[below]at(8.1,-5){$\substack{[3^2]\times\{5\}\\\times7\times{13}}$};
\end{tikzpicture}
\caption{An odd covering with $15$ used exactly four times as a modulus}
\label{fig:four15s}
\end{figure}
\end{proof}

\begin{theorem}\label{thm:five21s}
There exists a covering system of the integers such that all moduli are odd, greater than 1, and distinct except that the modulus $21$ is used five times.
\end{theorem}
\begin{proof}
A tree diagram of such a covering system is given by Figures~\ref{fig:five21s}-\ref{fig:five21sT3}, with the main tree in Figure~\ref{fig:five21s} and subtrees in Figures~\ref{fig:five21sT1}-\ref{fig:five21sT3}. Here, $q>31$ is a prime.
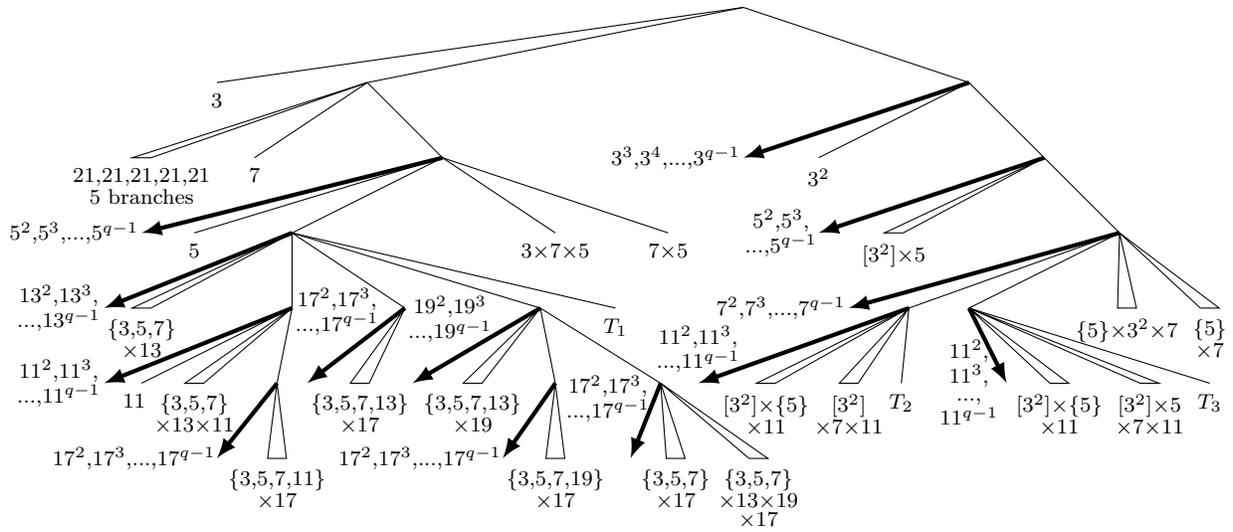
\begin{figure}[H]
\centering
\begin{tikzpicture}[scale=1]
\draw(0,0)--(-7,-1);\node[below]at(-7,-1){$\substack{3}$};
\draw(0,0)--(-5,-1);\draw(-5,-1)--(-8.125,-2)--(-7.875,-2)--(-5,-1);\node[below]at(-8,-2){$\substack{21,21,21,21,21\\5\text{ branches}}$};
\draw(-5,-1)--(-6.5,-2);\node[below]at(-6.5,-2){$\substack{7}$};

\draw(-5,-1)--(-4,-2);\draw[ultra thick,-latex](-4,-2)--(-8,-3);\node at(-8.9,-3){$\substack{5^2,5^3,\dotsc,5^{q-1}}$};
\draw(-4,-2)--(-7.3,-3);\node[below]at(-7.3,-3){$\substack{5}$};
\draw(-4,-2)--(-6,-3);\draw[ultra thick,-latex](-6,-3)--(-8.5,-4);\node at(-9.1,-4){$\substack{13^2,13^3,\\\dotsc,13^{q-1}}$};
\draw(-6,-3)--(-8.125,-4)--(-7.875,-4)--(-6,-3);\node[below]at(-8,-4){$\substack{\{3,5,7\}\\\times13}$};

\draw(-6,-3)--(-6,-4);\draw[ultra thick,-latex](-6,-4)--(-8.5,-5);\node at(-9.1,-5){$\substack{11^2,11^3,\\\dotsc,11^{q-1}}$};
\draw(-6,-4)--(-8,-5);\node[below]at(-8.1,-5){$\substack{11}$};
\draw(-6,-4)--(-7.425,-5)--(-7.175,-5)--(-6,-4);\node[below]at(-7.3,-5){$\substack{\{3,5,7\}\\\times13\times11}$};
\draw(-6,-4)--(-6.2,-5);\draw[ultra thick,-latex](-6.2,-5)--(-7,-6);\node at(-8.1,-6){$\substack{17^2,17^3,\dotsc,17^{q-1}}$};
\draw(-6.2,-5)--(-6.325,-6)--(-6.075,-6)--(-6.2,-5);\node[below]at(-6.2,-6){$\substack{\{3,5,7,11\}\\\times17}$};
%\draw(-6,-5)--(-5.525,-6)--(-5.275,-6)--(-6,-5);\node[below]at(-5.4,-6){$\substack{\{3,5,7\}\\\times11\times17}$};

\draw(-6,-3)--(-4.5,-4);\draw[ultra thick,-latex](-4.5,-4)--(-5.8,-5);\node[above]at(-5.4,-4.5){$\substack{17^2,17^3,\\\dotsc,17^{q-1}}$};
\draw(-4.5,-4)--(-5.225,-5)--(-4.975,-5)--(-4.5,-4);\node[below]at(-5.1,-5){$\substack{\{3,5,7,13\}\\\times17}$};

\draw(-6,-3)--(-2.7,-4);\draw[ultra thick,-latex](-2.7,-4)--(-4.4,-5);\node[above]at(-3.9,-4.6){$\substack{19^2,19^3\\\dotsc,19^{q-1}}$};
\draw(-2.7,-4)--(-3.725,-5)--(-3.475,-5)--(-2.7,-4);\node[below]at(-3.6,-5){$\substack{\{3,5,7,13\}\\\times19}$};
\draw(-2.7,-4)--(-2.5,-5);\draw[ultra thick,-latex](-2.5,-5)--(-3.2,-6);\node at(-4.3,-6){$\substack{17^2,17^3,\dotsc,17^{q-1}}$};
\draw(-2.5,-5)--(-2.625,-6)--(-2.375,-6)--(-2.5,-5);\node[below]at(-2.5,-6){$\substack{\{3,5,7,19\}\\\times17}$};
\draw(-2.7,-4)--(-1.1,-5);\draw[ultra thick,-latex](-1.1,-5)--(-1.5,-6);\node at(-1.8,-5.2){$\substack{17^2,17^3,\\\dotsc,17^{q-1}}$};
\draw(-1.1,-5)--(-1.025,-6)--(-0.775,-6)--(-1.1,-5);\node[below]at(-0.9,-6){$\substack{\{3,5,7\}\\\times17}$};
\draw(-1.1,-5)--(0.075,-6)--(0.325,-6)--(-1.1,-5);\node[below]at(0.2,-6){$\substack{\{3,5,7\}\\\times13\times19\\\times17}$};

\draw(-6,-3)--(-1.7,-4);\node[below]at(-1.7,-4){$\substack{T_1}$};

\draw(-4,-2)--(-2.5,-3);\node[below]at(-2.5,-3){$\substack{3\times7\times5}$};
\draw(-4,-2)--(-1,-3);\node[below]at(-1,-3){$\substack{7\times5}$};

\draw(0,0)--(3,-1);\draw[ultra thick,-latex](3,-1)--(0,-2);\node at (-0.9,-2){$\substack{3^3,3^4,\dotsc,3^{q-1}}$};
\draw(3,-1)--(1,-2);\node[below]at(1,-2){$\substack{3^2}$};
\draw(3,-1)--(4,-2);\draw[ultra thick,-latex](4,-2)--(1,-3);\node at(0.5,-3){$\substack{5^2,5^3,\\\dotsc,5^{q-1}}$};
\draw(4,-2)--(1.875,-3)--(2.125,-3)--(4,-2);\node[below]at(2,-3){$\substack{[3^2]\times5}$};

\draw(4,-2)--(5,-3);\draw[ultra thick,-latex](5,-3)--(1.4,-4);\node at(0.5,-4){$\substack{7^2,7^3,\dotsc,7^{q-1}}$};
%\draw(5,-3)--(3.7,-4);\node[below]at(3.7,-4){$\substack{T_2}$};
%\draw(5,-3)--(4.4,-4);\node[below]at(4.4,-4){$\substack{T_3}$};

\draw(5,-3)--(2.2,-4);\draw[ultra thick,-latex](2.2,-4)--(-0.6,-5);\node[above]at(-0.6,-5){$\substack{11^2,11^3,\\\dotsc,11^{q-1}}$};
\draw(2.2,-4)--(0.175,-5)--(0.425,-5)--(2.2,-4);\node[below]at(0.3,-5){$\substack{[3^2]\times\{5\}\\\times11}$};
\draw(2.2,-4)--(1.275,-5)--(1.525,-5)--(2.2,-4);\node[below]at(1.4,-5){$\substack{[3^2]\\\times7\times11}$};
\draw(2.2,-4)--(2.1,-5);\node[below]at(2.1,-5){$\substack{T_2}$};

\draw(5,-3)--(3,-4);\draw[ultra thick,-latex](3,-4)--(3.5,-5);\node[above]at(3,-5.7){$\substack{11^2,\\11^3,\\\dotsc,\\11^{q-1}}$};
\draw(3,-4)--(4.075,-5)--(4.325,-5)--(3,-4);\node[below]at(4.2,-5){$\substack{[3^2]\times\{5\}\\\times11}$};
\draw(3,-4)--(5.275,-5)--(5.525,-5)--(3,-4);\node[below]at(5.4,-5){$\substack{[3^2]\times5\\\times7\times11}$};
\draw(3,-4)--(6.2,-5);\node[below]at(6.2,-5){$\substack{T_3}$};

\draw(5,-3)--(4.975,-4)--(5.225,-4)--(5,-3);\node[below]at(5.1,-4){$\substack{\{5\}\times3^2\times7}$};
\draw(5,-3)--(6.075,-4)--(6.325,-4)--(5,-3);\node[below]at(6.2,-4){$\substack{\{5\}\\\times7}$};
\end{tikzpicture}
\caption{An odd covering with $21$ used exactly five times as a modulus}
\label{fig:five21s}
\end{figure}

\begin{figure}[H]
\centering
\begin{tikzpicture}[scale=1]
\draw[ultra thick,-latex](0,0)--(-6,-1);\node at(-6.6,-1){$\substack{23^2,23^3,\\\dotsc,23^{q-1}}$};
\draw(0,0)--(-5.2,-1)--(-4.8,-1)--(0,0);\node[below]at(-5,-1){$\substack{\{3,5,7,13\}\times23}$};

\draw(0,0)--(-2.5,-1);\draw[ultra thick,-latex](-2.5,-1)--(-6,-2);\node at(-6.6,-2){$\substack{11^2,11^3,\\\dotsc,11^{q-1}}$};
\draw(-2.5,-1)--(-5.4,-2);\node[below]at(-5.5,-2){$\substack{11}$};
\draw(-2.5,-1)--(-4.725,-2)--(-4.475,-2)--(-2.5,-1);\node[below]at(-4.6,-2){$\substack{\{3,5,7\}\\\times23\times11}$};
\draw(-2.5,-1)--(-3.5,-2);\draw[ultra thick,-latex](-3.5,-2)--(-4.2,-3);\node at(-5.3,-3){$\substack{17^2,17^3,\dotsc,17^{q-1}}$};
\draw(-3.5,-2)--(-3.625,-3)--(-3.375,-3)--(-3.5,-2);\node[below]at(-3.5,-3){$\substack{\{3,5,7,11\}\\\times17}$};

\draw(0,0)--(0,-1);\draw[ultra thick,-latex](0,-1)--(-2.6,-2);\node[above]at(-1.3,-1.5){$\substack{11^2,11^3,\\\dotsc,11^{q-1}}$};
\draw(0,-1)--(-2,-2);\node[below]at(-2,-2){$\substack{11}$};
\draw(0,-1)--(-1.225,-2)--(-0.975,-2)--(0,-1);\node[below]at(-1.1,-2){$\substack{\{3,5,7\}\\\times13\times23\times11}$};
\draw(0,-1)--(0.5,-2);\draw[ultra thick,-latex](0.5,-2)--(-0.3,-3);\node at(-1.4,-3){$\substack{17^2,17^3,\dotsc,17^{q-1}}$};
\draw(0.5,-2)--(0.375,-3)--(0.625,-3)--(0.5,-2);\node[below]at(0.5,-3){$\substack{\{3,5,7,11\}\\\times17}$};

\draw(0,0)--(2,-1);\draw[ultra thick,-latex](2,-1)--(1,-2);\node[above]at(1,-1.6){$\substack{17^2,17^3,\\\dotsc,17^{q-1}}$};
\draw(2,-1)--(1.675,-2)--(1.925,-2)--(2,-1);\node[below]at(1.8,-2){$\substack{\{3,5,7,23\}\\\times17}$};

\draw(0,0)--(4.5,-1);\draw[ultra thick,-latex](4.5,-1)--(3,-2);\node[above]at(3.2,-1.7){$\substack{17^2,17^3,\\\dotsc,17^{q-1}}$};
\draw(4.5,-1)--(3.675,-2)--(3.925,-2)--(4.5,-1);\node[below]at(3.8,-2){$\substack{\{3,5,7\}\\\times17}$};
\draw(4.5,-1)--(5.075,-2)--(5.325,-2)--(4.5,-1);\node[below]at(5.2,-2){$\substack{\{3,5,7\}\\\times13\times23\times17}$};

\draw(0,0)--(5.5,-1);\node[below]at(5.5,-1){$\substack{T_4}$};
\draw(0,0)--(6.5,-1);\node[below]at(6.5,-1){$\substack{T_5}$};
\end{tikzpicture}
\caption{$T_1$ in Figure~\ref{fig:five21s}}
\label{fig:five21sT1}
\end{figure}

\begin{figure}[H]
\centering
\begin{minipage}{0.48\textwidth}
\centering
\begin{tikzpicture}[scale=1]
\draw[ultra thick,-latex](0,0)--(-3,-1);\node at(-3.6,-1){$\substack{19^2,19^3,\\\dotsc,19^{q-1}}$};
\draw(0,0)--(-2.425,-1)--(-2.175,-1)--(0,0);\node[below]at(-2.3,-1){$\substack{\{3,5,7,23\}\\\times19}$};
\draw(0,0)--(-1,-1);\draw[ultra thick,-latex](-1,-1)--(-2,-2);\node at(-3.1,-2){$\substack{17^2,17^3,\dotsc,17^{q-1}}$};
\draw(-1,-1)--(-1.325,-2)--(-1.075,-2)--(-1,-1);\node[below]at(-1.2,-2){$\substack{\{3,5,7,19\}\\\times17}$};
\draw(0,0)--(1,-1);\draw[ultra thick,-latex](1,-1)--(-0.5,-2);\node[above]at(0,-1.5){$\substack{17^2,17^3,\\\dotsc,17^{q-1}}$};
\draw(1,-1)--(0.075,-2)--(0.325,-2)--(1,-1);\node[below]at(0.2,-2){$\substack{\{3,5,7\}\\\times17}$};
\draw(1,-1)--(1.275,-2)--(1.525,-2)--(1,-1);\node[below]at(1.4,-2){$\substack{\{3,5,7\}\\\times23\times19\\\times17}$};
\end{tikzpicture}
\caption{$T_4$ in Figure~\ref{fig:five21sT1}}
\label{fig:five21sT4}
\end{minipage}
\begin{minipage}{0.48\textwidth}
\centering
\begin{tikzpicture}[scale=1]
\draw[ultra thick,-latex](0,0)--(-4,-1);\node at(-4.6,-1){$\substack{19^2,19^3,\\\dotsc,19^{q-1}}$};
\draw(0,0)--(-3.425,-1)--(-3.175,-1)--(0,0);\node[below]at(-3.3,-1){$\substack{\{3,5,7\}\\\times19}$};
\draw(0,0)--(-2.125,-1)--(-1.875,-1)--(0,0);\node[below]at(-2,-1){$\substack{\{3,5,7\}\\\times13\times23\times19}$};
\draw(0,0)--(-0.5,-1);\draw[ultra thick,-latex](-0.5,-1)--(-1.4,-2);\node at(-2.5,-2){$\substack{17^2,17^3,\dotsc,17^{q-1}}$};
\draw(-0.5,-1)--(-0.825,-2)--(-0.575,-2)--(-0.5,-1);\node[below]at(-0.7,-2){$\substack{\{3,5,7,19\}\\\times17}$};
\draw(0,0)--(1.5,-1);\draw[ultra thick,-latex](1.5,-1)--(0,-2);\node[above]at(0.5,-1.5){$\substack{17^2,17^3,\\\dotsc,17^{q-1}}$};
\draw(1.5,-1)--(0.575,-2)--(0.825,-2)--(1.5,-1);\node[below]at(0.7,-2){$\substack{\{3,5,7\}\\\times17}$};
\draw(1.5,-1)--(1.775,-2)--(2.025,-2)--(1.5,-1);\node[below]at(1.9,-2){$\substack{\{3,5,7\}\\\times13\times23\\\times19\times17}$};
\end{tikzpicture}
\caption{$T_5$ in Figure~\ref{fig:five21sT1}}
\label{fig:five21sT5}
\end{minipage}
\end{figure}

\begin{figure}[H]
\centering
\begin{tikzpicture}[scale=1]
\draw[ultra thick,-latex](0,0)--(-5,-1);\node at(-5.6,-1){$\substack{29^2,29^3,\\\dotsc,29^{q-1}}$};
\draw(0,0)--(-4.2,-1)--(-3.8,-1)--(0,0);\node[below]at(-4,-1){$\substack{[3^2]\times\{5,11,7\}\\\times29}$};

\draw(0,0)--(-2.5,-1);\draw[ultra thick,-latex](-2.5,-1)--(-3.5,-2);\node at(-4.6,-2){$\substack{13^2,13^3,\dotsc,13^{q-1}}$};
\draw(-2.5,-1)--(-2.825,-2)--(-2.575,-2)--(-2.5,-1);\node[below]at(-2.7,-2){$\substack{[3^2]\times\{5,7\}\\\times29\times13}$};

\draw(0,0)--(-0.5,-1);\draw[ultra thick,-latex](-0.5,-1)--(-1.5,-2);\node[above]at(-1.5,-1.7){$\substack{13^2,13^3,\\\dotsc,13^{q-1}}$};
\draw(-0.5,-1)--(-0.825,-2)--(-0.575,-2)--(-0.5,-1);\node[below]at(-0.7,-2){$\substack{[3^2]\times\{5,7\}\\\times11\times29\times13}$};

\draw(0,0)--(1.5,-1);\draw[ultra thick,-latex](1.5,-1)--(0.5,-2);\node[above]at(0.5,-1.7){$\substack{17^2,17^3,\\\dotsc,17^{q-1}}$};
\draw(1.5,-1)--(1.175,-2)--(1.425,-2)--(1.5,-1);\node[below]at(1.3,-2){$\substack{[3^2]\times\{5,7,11\}\\\times29\times17\\16\text{ branches}}$};

\draw(0,0)--(3.5,-1);\draw[ultra thick,-latex](3.5,-1)--(2.5,-2);\node[above]at(2.5,-1.7){$\substack{19^2,19^3,\\\dotsc,19^{q-1}}$};
\draw(3.5,-1)--(3.175,-2)--(3.425,-2)--(3.5,-1);\node[below]at(3.3,-2){$\substack{[3^2]\times\{5,7,11\}\\\times29\times19\\18\text{ branches}}$};
\end{tikzpicture}
\caption{$T_2$ in Figure~\ref{fig:five21s}}
\label{fig:five21sT2}
\end{figure}

\begin{figure}[H]
\centering
\begin{tikzpicture}[scale=1]
\draw[ultra thick,-latex](0,0)--(-6,-1);\node at(-6.6,-1){$\substack{31^2,31^3,\\\dotsc,31^{q-1}}$};
\draw(0,0)--(-5.2,-1)--(-4.8,-1)--(0,0);\node[below]at(-5,-1){$\substack{[3^2]\times\{5,7,11\}\\\times31}$};

\draw(0,0)--(-3.5,-1);\draw[ultra thick,-latex](-3.5,-1)--(-4.5,-2);\node at(-5.6,-2){$\substack{13^2,13^3,\dotsc,13^{q-1}}$};
\draw(-3.5,-1)--(-3.825,-2)--(-3.575,-2)--(-3.5,-1);\node[below]at(-3.7,-2){$\substack{[3^2]\times\{5,7\}\\\times31\times13}$};

\draw(0,0)--(-1.5,-1);\draw[ultra thick,-latex](-1.5,-1)--(-2.5,-2);\node[above]at(-2.5,-1.7){$\substack{13^2,13^3,\\\dotsc,13^{q-1}}$};
\draw(-1.5,-1)--(-1.825,-2)--(-1.575,-2)--(-1.5,-1);\node[below]at(-1.7,-2){$\substack{[3^2]\times\{5,7\}\\\times11\times31\times13}$};

\draw(0,0)--(0.5,-1);\draw[ultra thick,-latex](0.5,-1)--(-0.5,-2);\node[above]at(-0.5,-1.7){$\substack{17^2,17^3,\\\dotsc,17^{q-1}}$};
\draw(0.5,-1)--(0.175,-2)--(0.425,-2)--(0.5,-1);\node[below]at(0.3,-2){$\substack{[3^2]\times\{5,7,11\}\\\times31\times17\\16\text{ branches}}$};

\draw(0,0)--(2.5,-1);\draw[ultra thick,-latex](2.5,-1)--(1.5,-2);\node[above]at(1.5,-1.7){$\substack{19^2,19^3,\\\dotsc,19^{q-1}}$};
\draw(2.5,-1)--(2.175,-2)--(2.425,-2)--(2.5,-1);\node[below]at(2.3,-2){$\substack{[3^2]\times\{5,7,11\}\\\times31\times19\\18\text{ branches}}$};

\draw(0,0)--(4.5,-1);\draw[ultra thick,-latex](4.5,-1)--(3.5,-2);\node[above]at(3.5,-1.7){$\substack{23^2,23^3,\\\dotsc,23^{q-1}}$};
\draw(4.5,-1)--(4.175,-2)--(4.425,-2)--(4.5,-1);\node[below]at(4.3,-2){$\substack{[3^2]\times\{5,7,11\}\\\times31\times23\\22\text{ branches}}$};

\draw(0,0)--(6.5,-1);\draw[ultra thick,-latex](6.5,-1)--(5.5,-2);\node[above]at(5.5,-1.7){$\substack{29^2,29^3,\\\dotsc,29^{q-1}}$};
\draw(6.5,-1)--(5.875,-2)--(6.125,-2)--(6.5,-1);\node[below]at(6,-2){$\substack{[3^2]\\\times\{5,11\}\\\times29}$};
\draw(6.5,-1)--(7.375,-2)--(7.625,-2)--(6.5,-1);\node[below]at(7.5,-2){$\substack{[3^2]\times\{5,7,11\}\\\times31\times29\\16\text{ branches}}$};
\end{tikzpicture}
\caption{$T_3$ in Figure~\ref{fig:five21s}}
\label{fig:five21sT3}
\end{figure}
\end{proof}

\begin{theorem}\label{thm:eight25s}
There exists a covering system of the integers such that all moduli are odd, greater than 1, and distinct except that the modulus $25$ is used eight times.
\end{theorem}
\begin{proof}
A tree diagram of such a covering system is given by Figures~\ref{fig:eight25s}-\ref{fig:eight25sT4}, with the main tree in Figure~\ref{fig:eight25s} and subtrees in Figures~\ref{fig:eight25sT1}-\ref{fig:eight25sT4}. Here, $q>31$ is a prime.

\begin{figure}[H]
\centering
\begin{tikzpicture}[scale=1]
\draw[ultra thick,-latex](0,0)--(-5,-1);
\node[above]at(-3,-0.5){$\substack{3^2,3^3,\dotsc,3^{q-1}}$};
\draw(0,0)--(-3,-1);\node[below]at(-3,-1){$\substack{3}$};
\draw(0,0)--(1,-1);

\draw(1,-1)--(-6.75,-2)--(-6.25,-2)--(1,-1);\node[below]at(-6.5,-2){$\substack{\{3\}\times5}$};
\draw(1,-1)--(-5,-2)--(-7.625,-3)--(-7.375,-3)--(-5,-2);\node[below]at(-7.5,-3){$\substack{5^2,5^2,5^2,5^2,5^2\\5\text{ branches}}$};
\draw(1,-1)--(-4,-2)--(-5.625,-3)--(-5.375,-3)--(-4,-2);\node[below]at(-5.5,-3){$\substack{5^2,5^2,5^2\\3\text{ branches}}$};
\draw(-4,-2)--(-4.2,-3);\node[below]at(-4.2,-3){$\substack{3\times5^2}$};

\draw(-4,-2)--(0,-3);\draw[ultra thick,-latex](0,-3)--(-4,-4);
\node[above]at(-2.6,-3.6){$\substack{5^4,5^5,\dotsc,5^{q-1}}$};
\draw(0,-3)--(-3.125,-4)--(-2.875,-4)--(0,-3);\node[below]at(-3,-4){$\substack{\{3\}\times5^3}$};

\draw(0,-3)--(0,-4);
\draw[ultra thick,-latex](0,-4)--(-5,-5);\node[above]at(-5.2,-5){$\substack{7^2,7^3,\dotsc,7^{q-1}}$};
\draw(0,-4)--(-3.625,-5)--(-3.375,-5)--(0,-4);\node[below]at(-3.5,-5){$\substack{\{3,5^2\}\times7}$};

\draw(0,-4)--(-2,-5);
\draw[ultra thick,-latex](-2,-5)--(-3,-6);\node at(-4.1,-6){$\substack{11^2,11^3,\dotsc,11^{q-1}}$};
\draw(-2,-5)--(-2.125,-6)--(-1.875,-6)--(-2,-5);\node[below]at(-2,-6){$\substack{\{3,7,5^2\}\\\times11}$};
\draw(-2,-5)--(-0.625,-6)--(-0.375,-6)--(-2,-5);\node[below]at(-0.5,-6){$\substack{\{3\}\times5^3\times11}$};

\draw(0,-4)--(-0.5,-5);\node[below]at(-0.5,-5){$\substack{5^3\times7}$};

\draw(0,-3)--(5,-4);
\draw[ultra thick,-latex](5,-4)--(1,-5);\node[above]at(1.8,-4.7){$\substack{7^2,7^3,\dotsc,7^{q-1}}$};
\draw(5,-4)--(2.375,-5)--(2.625,-5)--(5,-4);\node[below]at(2.5,-5){$\substack{\{3,5^2\}\times7}$};

\draw(5,-4)--(4,-5);
\draw[ultra thick,-latex](4,-5)--(3,-6);\node at(1.9,-6){$\substack{11^2,11^3,\dotsc,11^{q-1}}$};
\draw(4,-5)--(3.875,-6)--(4.125,-6)--(4,-5);\node[below]at(4,-6){$\substack{\{3,7,5^2\}\\\times11}$};
\draw(4,-5)--(5.875,-6)--(6.125,-6)--(4,-5);\node[below]at(6,-6){$\substack{\{3\}\times5^3\times7\times11}$};

\draw(5,-4)--(6,-5);\node[below]at(6,-5){$\substack{3\times5^3\times7}$};

\draw(1,-1)--(6,-2);
\draw[ultra thick,-latex](6,-2)--(3,-3);\node[above]at(3.9,-2.6){$\substack{7^2,7^3,\dotsc,7^{q-1}}$};
\draw(6,-2)--(3.875,-3)--(4.125,-3)--(6,-2);\node[below]at(4,-3){$\substack{\{3,5\}\times7}$};
\draw(6,-2)--(5.5,-3);\node[below]at(5.5,-3){$\substack{T_1}$};
\draw(6,-2)--(6.5,-3);\node[below]at(6.5,-3){$\substack{T_2}$};
\end{tikzpicture}
\caption{An odd covering with $25$ used exactly eight times as a modulus}
\label{fig:eight25s}
\end{figure}
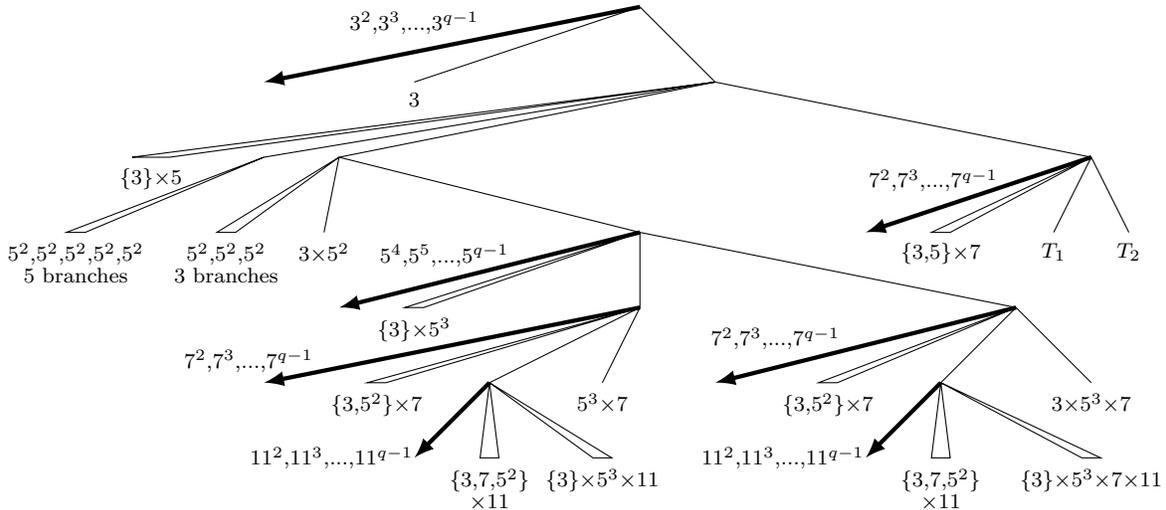

\begin{figure}[H]
\centering
\begin{tikzpicture}[scale=1]
\draw[ultra thick,-latex](-1,0)--(-6,-1);\node[above]at(-4,-0.5){$\substack{11^2,11^3,\dotsc,11^{q-1}}$};
\draw(-1,0)--(-5.125,-1)--(-4.875,-1)--(-1,0);\node[below]at(-5,-1){$\substack{\{3,7\}\times11}$};
\draw(-1,0)--(-3.125,-1)--(-2.875,-1)--(-1,0);\node[below]at(-3,-1){$\substack{\{3\}\times5\times7\times11}$};
\draw(-1,0)--(-1.125,-1)--(-0.875,-1)--(-1,0);\node[below]at(-1,-1){$\substack{\{3\}\times5\times11}$};

\draw(-1,0)--(1,-1);
\draw[ultra thick,-latex](1,-1)--(-2,-2);\node at(-3.1,-2){$\substack{17^2,17^3,\dotsc,17^{q-1}}$};
\draw(1,-1)--(-1.125,-2)--(-0.875,-2)--(1,-1);\node[below]at(-1,-2){$\substack{\{3,5,7,11\}\times17}$};

\draw(-1,0)--(4.5,-1);
\draw[ultra thick,-latex](4.5,-1)--(1.5,-2);\node[above]at(2.5,-1.5){$\substack{19^2,19^3,\dotsc,19^{q-1}}$};
\draw(4.5,-1)--(2.375,-2)--(2.625,-2)--(4.5,-1);\node[below]at(2.5,-2){$\substack{\{3,5,7,11\}\times19}$};

\draw(4.5,-1)--(4.5,-2);
\draw[ultra thick,-latex](4.5,-2)--(2.5,-3);\node at(1.6,-3){$\substack{5^3,5^4,\dotsc,5^{q-1}}$};
\draw(4.5,-2)--(3.375,-3)--(3.625,-3)--(4.5,-2);\node[below]at(3.5,-3){$\substack{\{3,7\}\times19\times5^2}$};

\draw(4.5,-1)--(7,-2);
\draw[ultra thick,-latex](7,-2)--(6,-3);\node[above]at(5.5,-2.8){$\substack{5^3,5^4,\dotsc,5^{q-1}}$};
\draw(7,-2)--(6.875,-3)--(7.125,-3)--(7,-2);\node[below]at(7,-3){$\substack{\{3,7\}\times11\times19\times5^2}$};
\end{tikzpicture}
\caption{$T_1$ in Figure~\ref{fig:eight25s}}
\label{fig:eight25sT1}
\end{figure}

\begin{figure}[H]
\centering
\begin{tikzpicture}[scale=1]
\draw[ultra thick,-latex](0,0)--(-7,-1);\node[above]at(-4.5,-0.5){$\substack{13^2,13^3,\dotsc,13^{q-1}}$};
\draw(0,0)--(-6,-1)--(-5.5,-1)--(0,0);\node[below]at(-6,-1){$\substack{\{3,5,7\}\times13}$};

\draw(0,0)--(-4.5,-1);
\draw[ultra thick,-latex](-4.5,-1)--(-6,-2);\node at(-6.5,-2){$\substack{5^3,5^4,\\\dotsc,5^{q-1}}$};
\draw(-4.5,-1)--(-5.425,-2)--(-5.175,-2)--(-4.5,-1);\node[below]at(-5.3,-2){$\substack{\{3,7\}\\\times13\times5^2}$};

\draw(0,0)--(-1.5,-1);
\draw[ultra thick,-latex](-1.5,-1)--(-4.7,-2);\node[above]at(-3.5,-1.7){$\substack{11^2,11^3,\\\dotsc,11^{q-1}}$};
\draw(-1.5,-1)--(-4.025,-2)--(-3.775,-2)--(-1.5,-1);\node[below]at(-4.1,-2){$\substack{\{3\}\times11}$};
\draw(-1.5,-1)--(-3.225,-2)--(-2.975,-2)--(-1.5,-1);\node[below]at(-3.1,-2){$\substack{\{3,5\}\\\times13\times11}$};
\draw(-1.5,-1)--(-2.225,-2)--(-1.975,-2)--(-1.5,-1);\node[below]at(-2.1,-2){$\substack{\{3\}\times5\\\times11}$};

\draw(-1.5,-1)--(-1.2,-2);
\draw[ultra thick,-latex](-1.2,-2)--(-1.9,-3);\node at(-2.8,-3){$\substack{5^3,5^4,\dotsc,5^{q-1}}$};
\draw(-1.2,-2)--(-1.525,-3)--(-1.275,-3)--(-1.2,-2);\node[below]at(-1.4,-3){$\substack{\{3,7\}\\\times13\times11\times5^2}$};
\draw(-1.5,-1)--(0.5,-2);
\draw[ultra thick,-latex](0.5,-2)--(-0.5,-3);\node[above]at(-0.5,-2.7){$\substack{29^2,29^3,\\\dotsc,29^{q-1}}$};
\draw(0.5,-2)--(0.175,-3)--(0.425,-3)--(0.5,-2);\node[below]at(0.3,-3){$\substack{\{3,5,7,11,13\}\\\times29\\28\text{ branches}}$};

\draw(0,0)--(4,-1);
\draw[ultra thick,-latex](4,-1)--(0.9,-2);\node[above]at(1.9,-1.5){$\substack{11^2,11^3,\dotsc,11^{q-1}}$};
\draw(4,-1)--(1.575,-2)--(1.825,-2)--(4,-1);\node[below]at(1.5,-2){$\substack{\{3\}\times11}$};
\draw(4,-1)--(2.575,-2)--(2.825,-2)--(4,-1);\node[below]at(2.7,-2){$\substack{\{3,5\}\times7\\\times13\times11}$};
\draw(4,-1)--(3.675,-2)--(3.925,-2)--(4,-1);\node[below]at(3.8,-2){$\substack{\{3\}\times5\\\times11}$};

\draw(4,-1)--(4.8,-2);
\draw[ultra thick,-latex](4.8,-2)--(3.8,-3);\node at(2.7,-3){$\substack{17^2,17^3,\dotsc,17^{q-1}}$};
\draw(4.8,-2)--(4.175,-3)--(4.425,-3)--(4.8,-2);\node[below]at(4.3,-3){$\substack{\{3,5,13,11\}\\\times17}$};

\draw(4,-1)--(7,-2);
\draw[ultra thick,-latex](7,-2)--(5.2,-3);\node[above]at(5.6,-2.7){$\substack{19^2,19^3,\\\dotsc,19^{q-1}}$};
\draw(7,-2)--(5.775,-3)--(6.025,-3)--(7,-2);\node[below]at(5.9,-3){$\substack{\{3,5,13,11\}\\\times19}$};

\draw(7,-2)--(7,-3);
\draw[ultra thick,-latex](7,-3)--(6.3,-4);\node at(5.4,-4){$\substack{5^3,5^4,\dotsc,5^{q-1}}$};
\draw(7,-3)--(6.675,-4)--(6.925,-4)--(7,-3);\node[below]at(6.8,-4){$\substack{\{3,13\}\\\times19\times5^2}$};

\draw(7,-2)--(8.3,-3);
\draw[ultra thick,-latex](8.3,-3)--(7.6,-4);\node[above]at(7.6,-3.6){$\substack{5^3,5^4,\\\dotsc,5^{q-1}}$};
\draw(8.3,-3)--(7.975,-4)--(8.225,-4)--(8.3,-3);\node[below]at(8.1,-4){$\substack{\{3,13\}\times11\\\times19\times5^2}$};

\draw(0,0)--(7,-1);
\node[below]at(7,-1){$\substack{T_3}$};
\end{tikzpicture}
\caption{$T_2$ in Figure~\ref{fig:eight25s}}
\label{fig:eight25sT2}
\end{figure}

\begin{figure}[H]
\centering
\begin{tikzpicture}[scale=1]
\draw[ultra thick,-latex](0,0)--(-7,-1);\node[above]at(-4.5,-0.5){$\substack{23^2,23^3,\dotsc,23^{q-1}}$};
\draw(0,0)--(-6.3,-1)--(-5.8,-1)--(0,0);\node[below]at(-6.4,-1){$\substack{\{3,5,7,13\}\\\times23}$};

\draw(0,0)--(-5,-1);
\draw[ultra thick,-latex](-5,-1)--(-6.5,-2);\node at(-7,-2){$\substack{5^3,5^4,\\\dotsc,5^{q-1}}$};
\draw(-5,-1)--(-6.125,-2)--(-5.875,-2)--(-5,-1);\node[below]at(-6,-2){$\substack{\{3,7\}\\\times23\times5^2}$};

\draw(0,0)--(-3.5,-1);
\draw[ultra thick,-latex](-3.5,-1)--(-4.7,-2);\node[above]at(-4.6,-1.8){$\substack{5^3,5^4,\\\dotsc,5^{q-1}}$};
\draw(-3.5,-1)--(-4.325,-2)--(-4.075,-2)--(-3.5,-1);\node[below]at(-4.2,-2){$\substack{\{3,7\}\times13\\\times23\times5^2}$};

\draw(0,0)--(-1.8,-1);
\draw[ultra thick,-latex](-1.8,-1)--(-3,-2);\node[above]at(-3,-1.8){$\substack{31^2,31^3,\\\dotsc,31^{q-1}}$};
\draw(-1.8,-1)--(-2.625,-2)--(-2.375,-2)--(-1.8,-1);\node[below]at(-2.5,-2){$\substack{\{3,5,7,13,23\}\\\times31\\30\text{ branches}}$};

\draw(0,0)--(0,-1);
\draw[ultra thick,-latex](0,-1)--(-1.5,-2);\node[above]at(-1.2,-1.7){$\substack{11^2,11^3,\\\dotsc,11^{q-1}}$};
\draw(0,-1)--(-1.025,-2)--(-0.775,-2)--(0,-1);\node[below]at(-0.9,-2){$\substack{\{3\}\times11}$};
\draw(0,-1)--(-0.025,-2)--(0.225,-2)--(0,-1);\node[below]at(0.1,-2){$\substack{\{3,5\}\\\times23\times11}$};
\draw(0,-1)--(0.925,-2)--(1.225,-2)--(0,-1);\node[below]at(1.1,-2){$\substack{\{3\}\times5\\\times11}$};
\draw(0,-1)--(2.025,-2)--(2.325,-2)--(0,-1);\node[below]at(2.2,-2){$\substack{\{3\}\times13\\\times23\times11}$};

\draw(0,0)--(4,-1);
\draw[ultra thick,-latex](4,-1)--(2.8,-2);\node[above]at(2.9,-1.7){$\substack{11^2,11^3,\\\dotsc,11^{q-1}}$};
\draw(4,-1)--(3.375,-2)--(3.625,-2)--(4,-1);\node[below]at(3.5,-2){$\substack{\{3\}\times11}$};
\draw(4,-1)--(4.375,-2)--(4.625,-2)--(4,-1);\node[below]at(4.5,-2){$\substack{\{3,5\}\\\times7\times23\\\times11}$};
\draw(4,-1)--(5.375,-2)--(5.625,-2)--(4,-1);\node[below]at(5.5,-2){$\substack{\{3\}\times5\\\times11}$};
\draw(4,-1)--(6.475,-2)--(6.725,-2)--(4,-1);\node[below]at(6.6,-2){$\substack{\{3\}\times5\\\times13\times23\\\times11}$};

\draw(0,0)--(6.5,-1);
\node[below]at(6.5,-1){$\substack{T_4}$};
\end{tikzpicture}
\caption{$T_3$ in Figure~\ref{fig:eight25sT2}}
\label{fig:eight25sT3}
\end{figure}

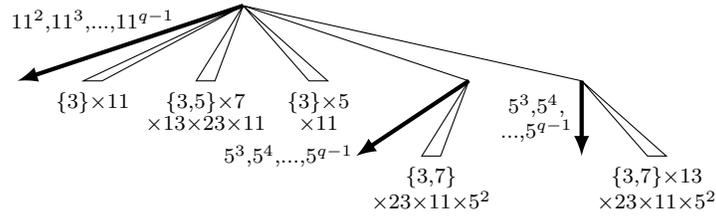
\begin{figure}[H]
\centering
\begin{tikzpicture}[scale=1]
\draw[ultra thick,-latex](0,0)--(-3,-1);\node[above]at(-2,-0.5){$\substack{11^2,11^3,\dotsc,11^{q-1}}$};
\draw(0,0)--(-2.125,-1)--(-1.875,-1)--(0,0);\node[below]at(-2,-1){$\substack{\{3\}\times11}$};
\draw(0,0)--(-0.625,-1)--(-0.375,-1)--(0,0);\node[below]at(-0.5,-1){$\substack{\{3,5\}\times7\\\times13\times23\times11}$};
\draw(0,0)--(0.875,-1)--(1.125,-1)--(0,0);\node[below]at(1,-1){$\substack{\{3\}\times5\\\times11}$};

\draw(0,0)--(3,-1);
\draw[ultra thick,-latex](3,-1)--(1.5,-2);\node at(0.6,-2){$\substack{5^3,5^4,\dotsc,5^{q-1}}$};
\draw(3,-1)--(2.375,-2)--(2.625,-2)--(3,-1);\node[below]at(2.5,-2){$\substack{\{3,7\}\\\times23\times11\times5^2}$};

\draw(0,0)--(4.5,-1);
\draw[ultra thick,-latex](4.5,-1)--(4.5,-2);\node at(3.9,-1.5){$\substack{5^3,5^4,\\\dotsc,5^{q-1}}$};
\draw(4.5,-1)--(5.375,-2)--(5.625,-2)--(4.5,-1);\node[below]at(5.5,-2){$\substack{\{3,7\}\times13\\\times23\times11\times5^2}$};
\end{tikzpicture}
\caption{$T_4$ in Figure~\ref{fig:eight25sT3}}
\label{fig:eight25sT4}
\end{figure}
\end{proof}

\section{Covering Subsets of the Integers}\label{sec:coveringsubsets}
In Section~\ref{sec:introduction}, we briefly discussed how covering systems with repeated moduli can be applied to derive odd coverings of subsets of the integers. The following theorem demonstrates such an application of the covering system presented in the proof of Theorem~\ref{thm:three9s}.

\begin{theorem}\label{thm:mod9covering}
Let $0\leq j\leq 8$ be an integer, and let $S_j=\{a\in\mathbb{Z}:a\equiv j+3\pmod{9}\text{ or }a\equiv j-3\pmod{9}\}$. Any subset of the integers with only finitely many terms in $S_j$ has an odd covering.
\end{theorem}
\begin{proof}Notice that the covering system presented in Theorem~\ref{thm:three9s} is of the form $\{0\pmod{9},3\pmod{9},6\pmod{9}\}\cup\{r_i\pmod{m_i}:1\leq i\leq\upsilon\}$, where all moduli $m_i$ are odd, distinct, greater than $1$, and not equal to $9$. By Lemma~\ref{lem:coveringshift}, $\mathfrak{C}_j=\{j\pmod{9}\}\cup\{r_i+j\pmod{m_i}:1\leq i\leq\upsilon\}$ is an odd covering of the set $\Z\setminus S_j$.

Let $T\subseteq\Z$ be such that $T\cap S_j=\{a_\ell:1\leq\ell\leq\tau\}$ for some nonnegative integer $\tau$. Then $\{j\pmod{9}\}\cup\{r_i+j\pmod{m_i}\}\cup\{a_\ell\pmod{M+2\ell}\}$ is an odd covering of $T$, where $M$ is the maximum modulus among $m_i$.
\end{proof}

By considering $j=0$ in Theorem~\ref{thm:mod9covering}, we obtain the following corollary.

\begin{corollary}\label{cor:v_3>2}
Any subset of the integers with only finitely many terms not in $\{a\in\Z:\text{if }3\mid a\text{ then }3^2\mid a\}$ has an odd covering.
\end{corollary}

Before providing a proof of Theorem~\ref{thm:integersequences}, we first discuss the existence of an odd covering of the perfect numbers and an odd covering of the Fermat numbers. Recall that a positive integer $n$ is called a \emph{perfect number} if the sum of the proper divisors of $n$ is equal to $n$. First proven by Euler, it is well known that every even perfect number has the form $2^{p-1}(2^p-1)$, where $p$ and $2^p-1$ are both prime. Currently, all known perfect numbers are even and are therefore of this form. Note that for an odd prime $p$, $2^{p-1}(2^p-1)\equiv 1\pmod{3}$. Although it is not known whether any odd perfect numbers exist, Touchard in 1953 \cite{touchard} showed that any odd perfect number $n$ must satisfy either $n\equiv 1\pmod{12}$ or $n\equiv 9\pmod{36}$. Thus, we deduce that if $n$ is a perfect number, then $n\not\equiv 2\pmod{3}$. Since $t_3\leq 2$, it follows from Lemma~\ref{lem:coveringshift}, along with our discussion preceding Theorem~\ref{thm:integersequences}, that the set of perfect numbers has an odd covering.

Next, recall that a \emph{Fermat number} is of the form $2^{2^a}+1$, where $a$ is a positive integer. All Fermat numbers are congruent to $2$ modulo $3$, and therefore the set of Fermat numbers has an odd covering. While the set of perfect numbers and the set Fermat numbers each has an odd covering that stems from $t_3\leq 2$, an odd covering of the union of these two sets does not immediately follow from that result. On the other hand, since any perfect number that satisfies $0\pmod{3}$ must also satisfy $0\pmod{9}$, we see that the union of perfect numbers and Fermat numbers has an odd covering by Corollary~\ref{cor:v_3>2}. With this observation, we are now ready to prove Theorem~\ref{thm:integersequences}.

\begin{proof}[Proof of Theorem~$\ref{thm:integersequences}$]
Let $S=\{a\in\Z:\text{if }3\mid a\text{ then }3^2\mid a\}$. The square of an integer is congruent to $0$, $1$, $4$, or $7$ modulo $9$. Thus, every sum of two squares is congruent to $0$, $1$, $2$, $4$, $5$, $7$, or $8$ modulo $9$ and is therefore in $S$. Similarly, every sum of two cubes is in $S$ since the cube of an integer is congruent to $0$, $1$, or $8$ modulo $9$, and hence, every sum of two cubes is congruent to $0$, $1$, $2$, $7$, or $8$ modulo $9$.

Any powerful number that is divisible by $3$ is, by definition, also divisible by $9$, and is therefore trivially in $S$. As for primes, the only prime that is divisible by $3$ is $3$ itself. Since powers of primes are powerful, all primes and powers of primes are elements of $S$.

It is known that the number of derangements of $n$ objects is given by $d_n=n!\sum_{i=0}^n\frac{(-1)^i}{i!}$. If $n\equiv0\pmod{3}$, then $d_n\equiv n!\frac{(-1)^n}{n!}\not\equiv0\pmod{3}$. If $n\equiv1\pmod{3}$, then 
\begin{align*}
d_n&\equiv n!\sum_{i=n-4}^n\frac{(-1)^i}{i!}\\
&=(-1)^n(n(n-1)(n-2)(n-3)-n(n-1)(n-2)+n(n-1)-n+1)\\
&=(-1)^n(n(n-1)(n-2)(n-4)+(n-1)^2)\\
&\equiv0\pmod{9}.
\end{align*}
If $n\equiv2\pmod{3}$, then
\begin{align*}
d_n&\equiv n!\sum_{i=n-3}^n\frac{(-1)^i}{i!}\\
&=(-1)^{n-1}(n(n-1)(n-2)-n(n-1)+n-1)\\
&=(-1)^{n-1}(n-1)(n^2-3n+1)\\
&\not\equiv0\pmod{3}.
\end{align*}
Therefore, every number of derangements is in $S$.

Since our discussion preceding this proof established that the set of perfect numbers and the set of Fermat numbers are subsets of $S$, the theorem now follows from Corollary~\ref{cor:v_3>2}.
\end{proof}

Theorem~\ref{thm:mod9covering} demonstrates the usefulness of the covering system presented in the proof of Theorem~\ref{thm:three9s} to the study of odd coverings of subsets of the integers. The application, however, is dependent on the residues of the repeated moduli relative to each other. For instance, it is unknown whether there exists a covering system with the congruences $4\pmod{9}$, $5\pmod{9}$, and $r\pmod{9}$ for some $0\leq r\leq 3$ or $6\leq r\leq8$, such that all other moduli are odd, distinct, greater than 1, and not equal to $9$. If such a covering system exists, then the set of the sums of three cubes will have an odd covering. This result does not follow from the covering system presented in the proof of Theorem~\ref{thm:three9s}.  Another obvious direction for improvement is to establish $t_9\leq 2$, which would also demonstrate an odd covering of other combinatorial sequences such as the Bell numbers.

\section{Acknowledgements}

These results are based on work supported by the National Science Foundation under grant numbered MPS-2150299.

\end{document}